\documentclass{siamart190516}

\def\ao#1{{#1}}
\def\aor#1{{ #1}}

\usepackage{lipsum}
\usepackage{amsfonts}
\usepackage{graphicx}
\usepackage{epstopdf}
\usepackage{algorithmic}
\ifpdf
  \DeclareGraphicsExtensions{.eps,.pdf,.png,.jpg}
\else
  \DeclareGraphicsExtensions{.eps}
\fi

\newsiamremark{remark}{Remark}
\newsiamremark{hypothesis}{Hypothesis}
\crefname{hypothesis}{Hypothesis}{Hypotheses}
\newsiamthm{claim}{Claim}

\headers{A Small Gain Analysis of Actor Critic}{A. Olshevsky and B. Gharesifard}

\title{A Small-Gain Analysis of Single Timescale Actor Critic
}

\author{Alex Olshevsky\thanks{Department of Electrical and Computer Engineering and Division of System Engineering, Boston University,
  (\email{alexols@bu.edu}, \url{https://sites.bu.edu/aolshevsky}).}
\and Bahman Gharesifard\thanks{Department of Electrical and Computer Engineering, University of California at Los Angeles,  
  (\email{gharesifard@ucla.edu}, \url{https://gharesifard.github.io}).}
}

\usepackage{amsopn}

\usepackage{prerex}
\usetikzlibrary{fit}
\ifpdf
\hypersetup{
  pdftitle={A Small-Gain Analysis of Single Timescale Actor Critic},
  pdfauthor={A. Olshevsky and B. Gharesifard}
}
\fi

\newtheorem{assumption}{Assumption}

\newcommand\subscr[2]{#1_{\textup{#2}}}

\newcommand{\pder}[2]{\frac{\partial #1}{\partial #2}}

\newcommand{\pdertwom}[3]{\frac{\partial^2 #1}{\partial #2\partial #3}}

\begin{document}

\maketitle

\begin{abstract}
We consider a version of actor-critic which uses proportional step-sizes and only one critic update with a single sample from the stationary distribution per actor step. We provide an analysis of this method using the small-gain theorem. Specifically, we prove that this method can be used to find a stationary point, and that the resulting sample complexity improves the state of the art for actor-critic methods to $O \left(\mu^{-2} \epsilon^{-2} \right)$ to find an $\epsilon$-approximate stationary point where $\mu$ is the condition number associated with the critic.  
\end{abstract}




\section{Introduction}

Since their introduction  decades ago in~\cite{barto1983neuronlike}, actor-critic methods have emerged as a popular approach to variety of reinforcement learning problems. Actor-only methods, such as policy gradient, seek to optimize a policy from a given class using noisy samples of the gradient. Estimating the gradient of a policy, however, is challenging because it requires knowing the true Q-values corresponding to the policy. While policy gradient methods usually estimate these using sample path rollout(s), actor-critic method adds a second parallel iteration designed to estimate the Q-values. It is sometimes asserted that, for a number of problems, actor-critic methods can outperform policy gradient due to the potentially high variance in sample path rollouts.

Possibly because of this, variants on actor-critic have achieved state-of-the-art performance on a number of benchmarks. For example, \cite{lazaridis2020deep} performs a comparison of several popular RL methods on a variety Atari games, finding that in the human-start regime an actor-critic variant comes first.

Most previous theoretical analyses of actor critic were based on two variations on the same idea. The first variation is to slow down the actor by giving the actor a step-size that decays slower than the critic (e.g., \cite{konda1999actor, konda2003onactor}). With this strategy,  we can use a perturbed version of the analysis of Temporal Difference (TD) methods to argue that the critic always has a good approximation of the correct Q-values. With such step-sizes, actor-critic can then be thought of as a  gradient method with error. 

An alternative approach is to do many critic steps per actor step. This can ensure that the actor always has good approximations of the correct Q-values coming from the critic, and, once again, this allows an analysis of actor-critic as a gradient method with error. It should be emphasized, however, that this method can outperform vanilla policy gradient methods because the variance from this procedure can be much better than than the variance from trajectory rollouts (a point made in \cite{xu2020improved}).

However, it is not clear that actor-critic really needs either of  these approaches, and indeed they are often not used in practice. Either approach is inconvenient, {\ao since both involve artificially slowing down the actor, and this results in slower convergence rates.  }

{ \ao Indeed, as we explain more formally later in the paper, actor-critic is an attempt to perform approximate iterations of stochastic gradient descent (SGD) on a non-convex objective.  The choice of step-size is crucial here. For non-convex SGD, a simple argument shows that a step-size of  $1/T^{u}$ multiplying the gradient at each step translates into a convergence rate of $\max(1/T^{u}, 1/T^{1-u})$ \cite{nonconvexsgd}. Informally speaking, the  first term of the maximum comes from analyzing the expected progress towards the optimal solution  at each step, while the second comes from the variance of each update. The best choice is $u=1/2$ which trades off the best between these and clearly any other choice of $u \in [0,1]$ results in a worse scaling of the maximum above. In actor-critic the same  tradeoff reappears: slowing down the actor relative to the best $T^{-1/2}$ step-size decay to make it possible for the critic to ``catch up'' results in a worse decay with $T$.
}

The underlying technical challenge in getting rid of these approaches is that, without two timescales or having the critic take many steps per actor step, one cannot immediately argue that the actor has access to approximately correct $Q$-values from the critic. As a result, it becomes much more challenging to connect actor-critic to perturbed gradient descent. 
Our work seeks to address this by showing how a version of the small gain theorem can be used to give an analysis of actor-critic without either of the above approaches.

\subsection{Literature Review} 

Because actor-critic methods are so widely used in reinforcement learning, recent years have seen several efforts to analyze their convergence. The most natural metric is the sample complexity, i.e., the number of samples needed to achieve a certain objective. In this case, the sample complexities are typically given in terms of finding a stationary point, i.e., one where the gradient squared of the actor is at most $\epsilon$, or the running averages of the squared norms of gradients is at most $\epsilon$. When the critic uses a linear approximation that cannot perfectly describe the actor's Q-values, the results below will give sample complexities until an approximate sttionary point, where the quality of the approximation depends on the accuracy that the critic is able to attain. We will use $O(\cdot)$ notation below, and this will be understood to hide logarithmic factors. Moreover, $\gamma$ denotes the discount factor of the problem (to be formally defined later). 

We begin by discussing the closely related works \cite{wu2020finite, shen2020asynchronous, kumar2019sample, xu2020improved}. All four of these papers gave a convergence analysis of actor-critic methods along with an associated sample complexities. The paper \cite{wu2020finite} studied a ``two time scale'' version where the actor and critic update using step-sizes that decay at different rates, attaining an $O(\epsilon^{-2.5})$ sample complexity. The paper \cite{shen2020asynchronous} gave a non-asymptotic analysis of an approximate version of actor-critic where an additional source of error proportional to $O(1-\gamma)$ was introduced in the final outcome. The algorithm also followed the two-timescale approach. This work also attained $O(\epsilon^{-2.5})$ sample complexity. The paper~\cite{kumar2019sample} considers several versions of actor-critic methods. While a sample complexity of $O(\epsilon^{-3})$ is obtained for one version of actor-critic, this was improved to $O(\epsilon^{-2.5})$ for an accelerated variation. These sample complexities required many critic steps per actor step. 

More recently, \cite{xu2020improving} was the first paper to attain a sample complexity $O(\epsilon^{-2})$. The algorithm used many critic steps per actor step. 
To explain the meaning of this result, note first that $O(\epsilon^{-2})$ is the rate one 
would obtain from gradient descent on a nonconvex function with noisy gradient evaluations. Thus if we did infinitely many critic steps per actor step, and then performed a standard  Stochastic Gradient Descent (SGD) analysis, we would immediately obtain an $O(\epsilon^{-2})$ rate. The work \cite{xu2020improving} attains the same decay rate with $\epsilon$, but with finitely many critic steps per actor step. 
We note that the interpretation of some of the results in  \cite{xu2020improving} was questioned in~\cite{khodadadian2021finite} (see Appendix C of that paper), but the critiques do not seem to apply to the actor-critic sample complexity we are discussing here. 

While writing this paper, we found the work \cite{chen2021closing} which is very closely related to the present paper: \cite{chen2021closing} gives an  analysis of single timescale actor-critic with $O(\epsilon^{-2})$ sample complexity. As far as we are aware, \cite{xu2020improving} and \cite{chen2021closing} are the only two papers to attain an $O(\epsilon^{-2})$ sample complexity for actor-critic; for these papers, we will thus also describe scaling with the condition number of the critic, to be formally defined in the main body of the paper, and which we denote by $\mu$. In terms of both $\epsilon$ and $\mu$ (and treating the remaining variables as constants), the scaling in \cite{xu2020improving} is $O \left(\mu^{-3} \epsilon^{-2}\right)$. The scaling in \cite{chen2021closing} is a little worse at  $O \left(\mu^{-4} \epsilon^{-2}\right)$.

Other related work considered algorithms building on the actor-critic framework. For example, the paper~\cite{hong2020two} considered a two-time scale version of the ``natural'' actor-critic method, where additional entropy regularizations are added to the update. Applied to a linearized version of the actor-critic objective, this gave an $O(\epsilon^{-4})$ sample complexity until that linearized objective is within $\epsilon$ of its optimal value. The paper~\cite{barakat2021analysis} considers a variation on actor critic by introducing a target update, with the result being  three different time scales for updates. The final sample complexity is $O(\epsilon^{-3})$. A similar sample complexity was given in the off-policy setting in~\cite{khodadadian2021finite}. We also mention  \cite{qiu2021finite}, which gave an iteration complexity analysis of actor-critic assuming the critic computes an approximation to the correct values per each actor step.  


For policy gradient methods, a sample complexity of $O(\epsilon^{-2})$ is known; see e.g., \cite{zhang2020global} for a modern analysis, and see also \cite{xu2020improved} for speedup. If more structure about the problem is known and one considers an exact version, one can sometimes obtain a better rate; for example, \cite{hambly2021policy} gives a geometric rate for an exact version of policy gradient when applied to a {\ao Linear-Quadratic Regulator (LQR)} problem. Likewise, information about about the structure of the optimal policy in general was shown to be useful in \cite{roy2021online}.

\subsection{Our contribution} 

Our  contribution is to use the small-gain theorem to analyze a single timescale actor critic method, i.e., a method which uses proportional step-sizes and a single critic step per actor step. We show that this method works in the same sense that all the previous methods work: it finds a stationary point when the critic's approximation is perfect, and an approximate stationary point otherwise.  We obtain a sample complexity of $O \left(\mu^{-2} \epsilon^{-2} \right)$. 

This compares favorably to the previous literature. One point of comparison is the now-classic work \cite{castro2010convergent}  also gave a
single timescale method; however, that method only converged to a neighborhood of a stationary solution even when the critic's approximation is perfect. As mentioned above, the only other analysis we know of single-timescale actor critic is from \cite{chen2021closing}, which gives a worse sample complexity of $O \left(\mu^{-4} \epsilon^{-2}\right)$. 
The main difference appears to be that where \cite{chen2021closing} use an explicit Lyapunov analysis, we instead rely on a small-gain analysis, which allows us to improve by a quadratic factor in $\mu^{-1}$. The key step in our approach is a nonlinear version of the small gain theorem which handles the complex dependencies between the actor-to-critic and critic-to-actor gains.

\section{Backround: Markov Decision Processes and Actor-Critic}

We consider a {\ao Markov Decision Process (MDP)} with a finite number of states and actions. When action $a$ is taken in state $s$, we will denote the probability that the next state is $s_{\rm next}$ by $P(s_{\rm next}|s,a)$, with $c(s,a)$ being the expected cost.   A policy is defined as a mapping from the set of states to the set of actions.  Defining $c_t$ to be the random variable representing the cost suffered at step $t$, our goal is to find a policy minimizing the infinite-horizon objective $E \left[ \sum_{t=0}^{\infty} \gamma^t c_t\right]$, where $\gamma \in (0,1)$ is the discount factor.

We will assume that we have parametrized the set of policies by $\theta$ with $\pi_{\theta}(a|s)$  being the probability of choosing action $a$ in state $s$. 
We   define the matrix $P_{\theta}$ to be the probability transition matrix among state-action pairs when actions are chosen according to $\pi_{\theta}$. We  define the value of a state $s$ as $J_{\theta}^*(s) = E_{\theta,s} \left[ \sum_{t=0}^{\infty} \gamma^t c_t\right]$, where the subscripts indicate that the starting state is $s$ and the policy followed is $\pi_{\theta}$. Similarly, we  define the $Q$-value as $Q_{\theta}^*(s,a) = E_{\theta,s,a} \left[ \sum_{t=0}^{\infty} \gamma^t c_t\right]$, where the subscripts indicate that one starts by taking action $a$ in state $s$ and follows the policy $\pi_{\theta}$ thereafter. We fix the initial distribution to some $\eta$ and consider the expectation 
\begin{equation} \label{eq:vdef}  V_{\theta}^{\eta} = \sum_{s} \eta(s) J_{\theta}^*(s), \end{equation} which is the expectation of the objective starting from  distribution $\eta$. 

We will use $p_{t,\theta}^{\eta}(s)$ to denote the probability that the state is $s$ at time $t$ if policy $\pi_{\theta}$ is followed after initializing the state from distribution $\eta$. The quantity $p_{t, \theta}^\eta$ will refer to the vector that stacks up $p_{t,\theta}^{\eta}(s)$ over all states $s$. Since $\eta$ will be fixed throughout this paper, we will usually omit it and simply write  $p_{t,\theta}(s)$, and likewise, we will simply write $V_{\theta}$ to represent the left-hand side of Eq.~\eqref{eq:vdef}. 

The policy gradient theorem~\cite{sutton2000policy} allows us to efficiently differentiate  $V_{\theta}$:
\begin{equation} \label{eq:pgt} \frac{d}{d \theta} V_{\theta}  = \sum_{s, t = 0, 1, \ldots} \gamma^t p_{t, \theta} (s) \sum_{a} \pi_{\theta}(a|s) Q_{\theta}^*(s,a) \frac{d}{d \theta} \log \pi_{\theta}(a|s). \end{equation}  
This identity allows us to efficiently apply a stochastic version of gradient descent as follows. We generate a random triple $(t, s_t, a_t)$ with probability $(1-\gamma) \gamma^t p_{t, \theta}(s_t) \pi_{\theta}(a_t|s_t)$ and update as
\begin{equation} \theta_{t+1} = \theta_t - \beta_t Q_{\theta_t}^*(s_t, a_t) \frac{d}{d \theta} \log  \pi_{\theta_t} (a_t|s_t), \label{eq:theta:ver1} 
\end{equation}
for an appropriately chosen step-size $\beta_t$. In the sequel, we will use $\nu_{\theta}$ to refer to the probability distribution over pairs $(s_t, a_t)$ induced by this sampling procedure.


The update of Eq.~\eqref{eq:theta:ver1} requires knowledge of the true $Q$-values. A policy gradient method will accomplish this with a trajectory rollout procedure. Specifically, a trajectory will be generated starting at state $(s_t,a_t)$ and the $Q$-value estimated by computing the empirical reward $\sum_{t} \gamma^t c_t$ on that trajectory. In practice, one typically attempts to save on sample complexity by generating just one trajectory and applying Eq.~\eqref{eq:theta:ver1} to each state-action pair while generating the empirical $Q$-value estimate by looking at the sub-trajectory beginning at $s_t,a_t$.  

By contrast, an actor-critic method will run a second update intended to produce estimates of the true $Q$-values. Since the number of state-action pairs $(s,a)$ is, in general, quite large, one typically attempts to find an approximation of the true $Q$-values.  We will assume that the critic will try to approximate the quantity $Q_{\theta}^*(s,a)$  linearly as 
\[ Q_{\theta}^*(s,a) \approx \sum_{i} \omega_i \phi_i(s,a), \] where each $\phi_i(\cdot, \cdot)$ can be thought of as extracting a feature from a state-action pair.  We can write this compactly as 
\[ Q_{\theta}^*(s,a) \approx \phi(s,a)^T \omega, \] where the vector $\phi(s,a)$ stacks up all the $\phi_i(s,a)$, and likewise $\omega$ stacks up the weights $\omega_i$. Moreover, by stacking up these vectors $\phi_i(s,a)$ as rows into a matrix $\Phi$, this can be written as 
\[ Q_{\theta}^* \approx \Phi \omega. \] Here $Q_{\theta}^*$ is a vector with as many entries as the number of state action pairs.

It is usually assumed that the best weights $\omega$ are unknown while a method to compute the features is available. To compute good weights, the critic will perform a TD(0) update by 
generating the state-action pair $s_t',a_t'$, then generating the next state and  action $s_t'', a_t''$ by following the MDP and the policy $\pi_{\theta}$ and suffering a cost of $c_t'$ in the process, and finally performing the update: 
\begin{equation} \label{eq:omega:ver1} \omega_{t+1} = \omega_t + \alpha_t \left(c_t' + \gamma \phi(s_t'', a_t'' )^T \omega_t - \phi(s_t',a_t')^T \omega_t \right) \phi(s_t', a_t'), \end{equation} 
for some appropriately chosen step-size $\alpha_t$. 
It is  important that $s_t', a_t'$ are sampled from the stationary distribution corresponding to policy $\pi_{\theta}$; otherwise, the TD(0) update above may diverge even for fixed $\theta$. When $s_t',a_t'$ are sampled from the stationary distribution, we will denote by $\mu_{\theta}$ the corresponding distribution over the four-tuple $(s_t',a_t', s_t'', a_t'')$.

In this paper, we will analyze an algorithm along these lines which we call {\em single-sample actor critic}. Its pseudocode is given in the algorithm box below. Note that it replicates Eq.~\eqref{eq:theta:ver1} and Eq.~\eqref{eq:omega:ver1} except that, in the actor's update of Eq.~\eqref{eq:theta:ver1} now uses approximation $\phi(s_t, a_t)^T \omega_t$ instead of the true $Q$-value.

We also add a projection onto the set $\Omega$ to the critic update, which we take to be a  compact convex set. Below, we discuss how to choose this set to be large enough so that it does not affect what the method is converging to. This is a standard tweak in actor-critic methods: it is a simple way to ensure a-priori that the iterates remain bounded. For example, the earlier papers \cite{wu2020finite, hong2020two, shen2020asynchronous, fu2020single} all add a projection, while \cite{barakat2021analysis} explicitly assume that the iterates remain bounded (which we can do as well without affecting any of the  results in this work). 

\begin{algorithm} 
\caption{Single sample actor critic}
\label{mainalg}
 Initialize at arbitrary $\theta_0, \omega_0$. 
 {\bf for} $t=1,2,\ldots$
 
 Generate a pair $(s_t, a_t)$ from $\nu_{\theta}$ and four-tuple $(s_t', a_t', s_t'', a_t'')$   from 
$\mu_{\theta}$ (independently from each other and from all past samples) and update as
\begin{align} \theta_{t+1} & =   \theta_t - \beta_t \left( \phi(s_t, a_t)^T \omega_t \right) \frac{d}{d \theta} \log \pi_{\theta} (a_t|s_t) \label{eq:theta:3} \\
 \omega_{t+1} & = P_{\Omega} \left[ \omega_t + \alpha_t  \left(c_t' + \gamma \phi(s_t'', a_t'')^T \omega_t  - \phi(s_t',a_t')^T \omega_t \right) \phi(s_t', a_t') \right] \label{eq:omega:3}
 \end{align}
\end{algorithm}

We remark that we call our method ``single sample actor critic'' because it only uses a single sample to update the actor and a single sample to update the critic at each step.  However, note that the samples here are not samples from the MDP directly: rather, we need to sample from the stationary distribution for the critic, and from the distribution $\nu_{\theta}$ for the actor. Multiple samples from the MDP will be needed to generate these single samples. 

Specifically,  this can be done by  having the critic generate a sufficiently long path from the MDP for mixing to occur and throw away all the samples except the last one. Likewise, the actor can generate a geometric random variable with density proportional to $\gamma^t$, sample from it, and generate path from the MDP with length equal to the sample thus generated. As a consequence, in order to translate any sample complexity obtained for this procedure into a sample complexity in terms of MDP samples, one needs multiply by $(1-\gamma)^{-1}$ plus an upper bound on the the mixing time of the policies.


\section{Assumptions} To analyze single-sample actor critic, we will need to make a number of assumptions. The assumptions we will make are quite standard in the literature, and will generally assume that a quantity encountered through the course of the algorithm is bounded. Our first assumption will assume that the  costs encountered throughout the algorithm do not blow up. {\ao This assumption helps ensure that various quantities appearing throughout the execution of the actor-critic method are naturally bounded.}

\begin{assumption}[Bounded Costs] There exists some finite $C_{\rm max}$ such that the costs $c_t$ belong to $[-C_{\rm max}, C_{\rm max}]$ with probability one. \label{ass:bcost}
\end{assumption}


Next, recall that the critic will generate samples from the stationary distribution from each policy. We need to ensure that the stationary distribution exists, and additionally, the convergence to it needs to be uniform over all $\theta$, as in the following assumption.  

\begin{assumption}[Uniform Mixing] Let $\rho_{t,\theta}$ be the distribution over state-action pairs after $t$ transitions of following policy $\pi_{\theta}$. There is a distribution $\rho_{\theta}$ and constants $C$ and $\lambda \in [0,1)$ such that 
\[ ||\rho_{t, \theta} - \rho_{\theta}||_1 \leq C \lambda^t. \]  \label{ass:stationary}
\end{assumption}

Note again the uniformity in the right-hand side above in that $C$ and $\lambda$ do not depend on $\theta$. Next, we will also need some regularity assumption on the parametrization of the policy. {\ao It goes without saying that optimization of smooth functions is easier than optimization of non-smooth functions and the following assumption ensures that various quantities appearing throughout the execution of actor-critic are smooth.}

\begin{assumption}[Smoothness of policy parametrization] The function $\pi_{\theta}(a|s)$ is a  \aor{thrice} continuously differentiable function of $\theta$ for all state-action pairs $s,a$. Further, there exist constants $K,K', K''$ such that for all $\theta_i,s,a$ we have
\[ \aor{||}\nabla_{\theta} \log \pi_{\theta}(a|s)\aor{||}  \leq  K, ~~
\aor{||}\nabla_{\theta} \pi_{\theta}(a|s)\aor{||}  \leq  K', ~~ 
 \aor{||}\nabla^2 \pi_{\theta}(a|s)\aor{||} \leq  K''.
\]  \label{ass:pi}
\end{assumption}

Next, observe that if the columns of the matrix $\Phi$ are linearly dependent, then we have some redundancy in the features: we can delete some features without altering the subspace that can be represented as a linear combination of feature vectors. It is of course helpful to assume the features are not redundant.
{ \ao If the features were redundant, then the critic might have multiple optimal solutions when approximating the value function corresponding to a policy, which complicates the analysis.} 
Additionally, it is also convenient to normalize the features so that all the feature vectors $\phi(s,a)$ have at most norm one. 

\begin{assumption}[Nonredundancy and norm of features] The matrix $\Phi$ is nonsingular and each of its rows has at most unit norm. \label{ass:features}
\end{assumption} 


Next, it is natural that the final accuracy of actor-critic depends on how well the critic is able to approximate the true $Q$-functions encountered in the course of the algorithms as linear combinations of features. {\ao For example, if the featurs are ``bad'' in that they fail to accurately approximate the $Q$-values associated with the actor policies, clearly actor-critic will perform poorly.} To that end, we have an assumption that provides a measure of this.

\begin{assumption}[Approximability of the true $Q$-values by the critic] 
Define $\omega_{\theta}$ to be the limit of temporal difference update of Eq.~\eqref{eq:omega:ver1} when the policy is fixed to $\pi_{\theta}$ and the projection is removed\footnote{That this limit exists was proven under Assumptions \ref{ass:bcost}, \ref{ass:stationary}, and \ref{ass:features} in \cite{tsitsiklis1996analysis}.}. Then for some $\delta  > 0$,  \[ \sup_{\theta} E_{s,a \sim \nu_{\theta}} ~ | Q_{\theta}^*(s,a) - \phi(s,a)^T \omega_{\theta} |  \leq \delta. \] \label{ass:deltadef}  \end{assumption} 

In words, $\delta$ determins how well the true $Q$-values are approximated by TD learning on the average. 
If the linear approximation always perfectly describes the true $Q$-values, then $\delta = 0$. As we will see, $\delta$ will appear in our convergence results below, as in the appropriate sense getting a final error in actor-critic that is too small compared to $\delta$ is not possible. 

Before making our next assumption, we need to introduce some new notation. Observe that it is possible to write~\eqref{eq:omega:3} as 
\begin{equation} \label{eq:omega:rewritten} \omega_{t+1} = P_{\Omega} \left( (I + \alpha_t  A_t) \omega_t - \alpha_t b_t) \right).
\end{equation} Indeed, inspecting~\eqref{eq:omega:3}, we have that
\begin{eqnarray}
A_t  =   \phi(s_t',a_t') (\gamma \phi(s_t'',a_t'') - \phi(s_t',a_t'))^T, ~~ 
b_t  = -c_t' \phi(s_t',a_t') \label{eq:at}  
\end{eqnarray}

Further, let us define $
A_{\theta}  =  E[A_t]$ and $
b_{\theta}  =  E[b_t]$. We emphasize the dependence on $\theta$ here because the expectation is being taken by generating the tuple $s_t',a_t',s_t'', a_t''$ from $\mu_{\theta}$. Since these samples are being generated i.i.d., there is no dependence on $t$. With this notation, our next assumption is as follows. 

\begin{assumption}[Exploration] There exists {\color{red} a $\mu \in (0,1)$} such that \label{ass:eigen}
\[ \sup_{\theta} \sup_{||x||=1} x^T A_{\theta} x \leq -\frac{\mu}{2} < 0. \] \label{ass:explore}
\end{assumption} 
The assumption is labeled as ``exploration'' because it holds if the policies $\pi_{\theta}$ explore all state-action pairs, as explained in Section \ref{sec:exploration} below. It is also a standard assumption: it is made in~\cite{wu2020finite, shen2020asynchronous, qiu2021finite}, and is closely related to assumptions which are made in \cite{kumar2019sample, barakat2021analysis}. 

Finally, we need to make sure our compact set $\Omega$ is large enough so that projection onto it does not introduce spurious minima. 

\begin{assumption}[Projection Set] The set $\Omega$ is a compact convex set which contains all of the vectors $\omega_{\theta}$ in its interior.  \label{ass:projection}
\end{assumption} 

{\ao The compactness of this set allows us to ensure that the critic iterates do not grow unbounded. The further assumption that $\omega_{\theta}$ is there to make sure that the quality of critic approximation is unaffected by the projection. }


\begin{remark} {\ao With one exception, there are no dependency relationships between the assumptions we make. In particular, it is possible for any subset of them to be satisfied or unsatisfied. The only exception to this is that Assumption 7 cannot even be stated without Assumption 2. Indeed, the matrix $A_{\theta}$ is defined by random sampling from the distribution $\mu_{\theta}$ (defined immediately after Eq. (2.4)), which is built from the stationary distribution of the policy $\pi_{\theta}$; Assumption 2 says that this stationary distribution is well-defined. }
\end{remark}

\subsection{The Exploration Assumption\label{sec:exploration}}

We now discuss the form of Assumption \ref{ass:explore}: specifically, we explain why we need to have a set of policies that have some exploration, and why that implies that Assumption \ref{ass:explore} must hold.

Indeed, suppose that a particular action $a$ in state $s$ has a large negative cost, so much so that the best policy must always take it. Any sample-based method needs to select that action at least once before it can incorporate this knowledge. However, glancing at Eq.~\eqref{eq:theta:3} we see this can fail to happen. Indeed, it may be that $\pi_{\theta}(a|s)=0$ for a set of $\theta$; then $(s,a)$ will never appear in Eq.~\eqref{eq:theta:3} and the update might always stay within the set of $\theta$ where $\pi_{\theta}(a|s)=0$. 

One way to overcome this  is to require that the policies in question have a positive probability of choosing every state action pair. For a finite-state MDP, this immediately implies Assumption \ref{ass:explore}, as the proposition below explains. 

\begin{proposition} Suppose Assumption \ref{ass:stationary} and Assumption \ref{ass:features} hold, and suppose further there exists some $p_{\rm min} > 0$ such that 
\[ \inf_{\theta, s, a} \pi_{\theta}(a|s) \geq p_{\rm min}. \] Then Assumption \ref{ass:explore} holds. \label{prop:exploration}
\end{proposition} 

\begin{proof} We will rely on the main result of \cite{liu2021temporal}, which interprets the mean TD(0) direction as a ``gradient splitting'' of an appropriately defined function and uses this to derive the following identity:
\begin{equation} \label{eq:quad} (\omega - \omega_{\theta})^T A_{\theta} (\omega - \omega_{\theta}) = -(1-\gamma) ||\Phi (\omega - \omega_{\theta})||_{\rm D}^2 - \gamma || \Phi (\omega - \omega_{\theta})||_{\rm Dir}^2, 
\end{equation} where, adopting the notation $\rho_{\theta}(s)$ for the stationary probability of state $s$ under the policy $\pi_{\theta}$, 
\begin{eqnarray*} 
||x||_{\rm D}^2 & =&  \sum_{s,a} \rho_{\theta}(s) \pi_{\theta}(a|s) x(s,a)^2 \\ 
||x||_{\rm Dir}^2 & = & \sum_{s,a,s',a'} \rho_{\theta}(s) \pi_{\theta}(a|s) P(s'|s,a) \pi_{\theta}(a'|s') (x(s,a) - x(s',a'))^2
\end{eqnarray*} Note that while the main result of \cite{liu2021temporal} was shown for TD(0), the proof applies immediately to TD(0) on the state action pairs. 
Now the assumptions clearly imply a uniform bound over all $\theta$ on how small $\rho_{\theta}(s) \pi_{\theta}(a|s)$ can be, which along with~\eqref{eq:quad} yield the result. 
\end{proof}


The parameter $\mu$ may be thought of as the condition number associated with the critic. Indeed, glancing at Eq. (\ref{eq:quad}) we see that the left-hand side measures the inner product between the $\omega - \omega_{\theta}$, the offset from the TD limit, and the TD direction $A(\omega - \omega_{\theta})$. The definition of $\mu$ bounds this as being below $-(\mu/2) ||\omega - \omega_{\theta}||^2$. Thus $\mu$ tells us how well the TD direction of the critic aligns with the vector pointing from $\omega$ to the final limit. 
It is thus not surprising that existing analysis of temporal difference learning under the optimally decaying step-size $1/t$ have  factors of $\mu^{-1}$ in the final bounds (e.g., \cite{bhandari2018finite, srikant2019finite, dalal2018finite}).

\section{Our Main   Result} 

We will use the short hand
\[ \Delta_t = \omega_t - \omega_{\theta_t} \] to denote the difference between the critic iterate at time $t$ and the ultimate limit of the critic update corresponding to the policy $\pi_{\theta_t}$. Further, it will be convenient to adopt the shorthand
\[ \nabla_t = \nabla V(\theta_t). \] Since the function $V(\theta)$ is not assumed to be convex, our goal is to argue that actor-critic finds a point such that $\nabla_t \approx 0$. A standard performance measure for this in non-convex optimization is the running average of the gradients, and it is typically argued that this is close to zero. Below, we find it convenient to make a slight modification, taking the average over the last half of the iterations. 

\begin{theorem} Suppose Assumptions \ref{ass:bcost} -- \ref{ass:projection} hold. We can choose $\alpha_t = 1/\sqrt{t}$ and $\beta_t = c/\sqrt{t}$, where $ c>0 $ is an appropriate constant chosen depending on the problem parameters such that the sequence of iterates produced by the Single Sample Actor Critic satisfies \begin{eqnarray*} \frac{1}{t/2} \sum_{k=t/2}^{t-1} ||\nabla_k||^2 & = & O \left( \delta^2 + \frac{\mu^{-1}}{\sqrt{t}} \right) 
\\ 
\frac{1}{t/2} \sum_{k=t/2}^{t-1} ||\Delta_k||^2 & = & O \left( \delta^2 + \frac{\mu^{-1}}{\sqrt{t}} \right) 
\end{eqnarray*}  where all \aor{quantities} except $\delta, \aor{\mu}, t$ are treated as constants in the $O(\cdot)$-notation. 

\label{thm:iid}
\end{theorem} 

In particular, if the critic can approximate the $Q$-values exactly, i.e., $\delta=0$, we have that the running average of the last half of the gradient norms squared approaches zero. Moreover, this running average is below $\epsilon$ after $O \left( \mu^{-2} \epsilon^{-2}\right)$ iterations. As discussed earlier, this sample complexity compares favorably to the scalings with $\mu$ and $\epsilon$ obtained in the previous literature.


Finally, if a specific point with gradient squared upper bounded by $\epsilon$ is sought, a simple trick is to take $t$ large enough so that the right-hand side of Theorem \ref{thm:iid} is below $\epsilon$, and then take a uniformly random iterate from the last $t/2$ iterates. Then, in expectation the squared norm of the resulting gradient will be at most $\epsilon$.

\begin{remark} {\ao In practice, step-sizes in actor-critic are almost always taken to be small and fixed on both the actor and the critic. In particular, as afar as we are aware, practitioners never choose a step-size on the actor that decays faster to zero than the step-size of the critic. A reader wishing to read about the details of a modern, state-of-the-art implementation of actor-critic methods can refer to  \cite{lazaridis2020deep}.}
\end{remark}

\subsection{Key ideas in the proof of Theorem \ref{thm:iid}} We will proceed by applying a version of the small gain theorem to the quantities of the left-hand side of Theorem \ref{thm:iid}, i.e., to the averages of $||\nabla_k||^2$ and $||\Delta_k||^2$.  In the ``easy'' direction, it is clear that if the error on the critic side is small (i.e., if the $||\Delta_k||$s are small), then the error on the actor side is small too; this follows via standard ``SGD with errors'' type analysis (indeed, we can view the actor as performing SGD with the error coming from the critic approximation).  In the more difficult direction, one needs to argue that if the actor is approximately close to stationarity, i.e., the $\nabla_k$ are small, then the errors of the critic $\Delta_k$ will also be small. If separate bounds conforming to this intuition can be established, one can attempt to put them together into an unconditional bound on both actor and critic using a small-gain type analysis.

Let us explain the challenges in deriving a bound on the critic error $\Delta_t$ in terms of the actor stationarity $\nabla_t$. If the actor simply followed gradient descent on the objective $V(\theta)$, then it would be fairly straightforward that small actor gradients imply small critic errors; this would follow since small gradients $\nabla_k$ means the actor moves little in expectation from step to step, so that the critic is approximately doing temporal difference steps. Thus one could simply apply a standard analysis of temporal difference learning with some added perturbations. 

Unfortunately, because the actor relies on the critic for its estimates of the $Q$-values from which it builds its estimate of the gradient direction, it is possible for the actor to be close to stationarity (i.e., $\nabla_k$ small) and yet move quite far from step to step (i.e., $\theta_{k+1} - \theta_k$ large). This dependency is what makes it challenging to argue that if the actor is close to stationarity, the critic errors are small. Ultimately, this is what results in a  nonlinear relationship between the gains involved and forces us to use a nonlinear small gain theorem for the analysis. We remark that our proof does not rely on the ODE method \cite{borkar2000ode, devraj2021fundamental}.



\section{Proof of our main result} 

We now begin a proof of our main result. Our first step is to reformulate the actor-critic update in terms of a coupled gradient-liked updates which has more of an optimization flavor. However, we first need to introduce some convenient notation. {\em From this point on, we assume  Assumptions \ref{ass:bcost} -- \ref{ass:projection} hold.}

\subsection{Notation} For convenience, we will introduce the notation $$Q_{\theta}(s,a) = \phi(s,a)^T \omega_{\theta},$$ where recall that $\omega_{\theta}$ was introduced in the statement of Assumption \ref{ass:deltadef}. Informally, $Q_{\theta}(s,a)$ will denote the ``ultimate TD approximation'' of the true $Q$-value $Q_{\theta}^*(s,a)$: it is what we would get if we fixed $\theta$ and let the critic take infinitely many steps. 

Moreover, we will use 
$$ Q_t(s,a) = \phi(s,a)^T \omega_t. $$ Informally, $Q_t(s,a)$ the estimate of the true $Q$ value $Q_{\theta}^*(s,a)$ which is obtained by the critic after $t$ steps of the actor-critic method. We stack these up into the vector $Q_t$ which has as many entries as the number of state-action pairs. Formally, $Q_t = \Phi \omega_t$. 

\subsection{Reformulating Actor-Critic}

We now give a re-formulation of the underlying problem in terms of two coupled gradient-like updates. We note that this section is not particularly new and such reformulations are common in all previous analysis of actor-critic. 

We begin by rewriting Eq.~(\ref{eq:omega:3}) as 
\[ \omega_{t+1} = P_{\Omega} \left[ \omega_{t} - \alpha_t \left( \widetilde{\nabla}(\theta_t, \omega_t) + w_{\rm c}(t) \right) \right]. \]
Here  $w_{\rm c}(t)$ is a random variable with zero expectation conditional on the entire past trajectory; and $\tilde{\nabla} (\theta_t, \omega_t)$ is the expected TD direction. More formally, we have that 
\[ \widetilde{\nabla}(\theta_t, \omega_t) = - A_{\theta_t} \omega_t + b_{\theta_t},  \] where  the quantities $A_{\theta}, b_{\theta}$ were defined earlier after Eq. (\ref{eq:at}). 

Note that, since $\omega_{\theta}$ is defined to be the limiting point of temporal diference learning (which exists due to the results of \cite{tsitsiklis1996analysis}), and is assumed to lie in the interior of the set $\Omega$, we have that  
\begin{equation} \label{eq:thetazero} \widetilde{\nabla} (\theta, \omega_{\theta}) = 0. \end{equation} An immediate consequence of this is that 
\begin{eqnarray*} \widetilde{\nabla}(\theta, \omega)^T (\omega - \omega_{\theta}) & = & (- A_{\theta} \omega + b_{\theta})^T (\omega - \omega_{\theta}) \\ 
& = & (- A_{\theta} \omega + b_{\theta})^T (\omega - \omega_{\theta}) - (- A_{\theta} \omega_{\theta} + b_{\theta})^T (\omega - \omega_{\theta}) \\ 
& \geq & \frac{\mu}{2} ||\omega - \omega_{\theta}||^2, 
\end{eqnarray*} where we used Eq. (\ref{eq:thetazero}) while the last inequality used Assumption \ref{ass:explore}. 


We next turn to  the actor update. It is immediate that we can rewrite Eq. (\ref{eq:theta:3}) as 
\[ \theta_{t+1}= \theta_t + \beta_t \left(  \sum_{s,k=0,\ldots} \gamma^k p_{k,\theta_t} (s) \sum_{a} \pi_{\theta_t}(a|s) Q_t(s,a) \frac{d}{d \theta} \log \pi_{\theta_t} (a|s)  + w_{\rm a}(t) \right), \]  where $w_{\rm a}(t)$ has zero expectation conditional on the entire past trajectory. Moreover, by the independent sampling by the actor and the critic, we have that  conditional on the past trajectory, $w_{\rm a}(t)$ and $w_{\rm c}(t)$ are independent. 

Let us further abstract this somewhat by defining 
\begin{eqnarray*} \hat{g}(\theta,Q) & =&   \sum_{s,t=0,\ldots} \gamma^t p_{t,\theta} (s) \sum_{a} \pi(a|s) Q(s,a) \frac{d}{d \theta} \log \pi_{\theta} (a|s) 
\end{eqnarray*} In terms of this notation, the policy gradient theorem can be reformulated as simply $\nabla V(\theta) = \hat{g}(\theta, Q_{\theta}^*)$. Further, we can rewrite the actor update as  
\[ \theta_{t+1} = \theta_{t} - \beta_t \left( \hat{g}(\theta_t, \Phi \omega_t) + w_{\rm a}(t) \right). \]

We summarize this discussion as the following proposition. 

\begin{proposition}[{\ao Reformulation of actor-critic}] The single-sample actor-critic update can be written as 
\begin{eqnarray} \theta_{t+1} & = & \theta_t - \beta_t \left( \hat{g}(\theta_t, \Phi \omega_t) + w_{\rm a}(t) \right) \label{eq:mainupdate} \\
\omega_{t+1} & = & \aor{P_{\Omega} \left( \omega_t - \alpha_t \left( \widetilde{\nabla} (\theta_t, \omega_t) + w_{\rm c}(t) \right) \right)}, \nonumber
\end{eqnarray} where if ${\mathcal F}_t$ is the $\sigma$-field generated by $w_{\rm a}(0), \ldots, w_{\rm a}(t-1), w_{\rm c}(0), \ldots, w_{\rm c}(t-1)$, then  \[ E [ w_{\rm a}(t) | {\mathcal F}_t ] = E[w_{\rm c}(t) | {\mathcal F}_t ] = 0,\] and further that $w_{\rm a}(t), w_{\rm c}(t)$ are conditionally independent given ${\mathcal F}_t$. Moreover, for all $\theta, \omega$, 
\begin{eqnarray} \nabla V(\theta) & = &   \hat{g}(\theta_, Q_{\theta}^*)  \label{eq:pgt2} \\ 
\widetilde{\nabla}(\theta,\omega)^T (\omega - \omega_{\theta}) & \geq & \frac{\mu}{2} ||\omega - \omega_{\theta}||^2,  \label{eq:critupdate1}
\end{eqnarray} for some strictly positive $\mu$. \label{prop:reformulation}
\end{proposition} 

We next introduce one more piece of notation. We will find it convenient to define 
\begin{equation} \label{eq:gdef} g(\theta, \omega) = \hat{g}(\theta, \Phi \omega), \end{equation} so that the actor-update can be written simply as 
\[ \theta_{t+1} = \theta_{t} - \beta_t \left( g(\theta_t,  \omega_t) + w_{\rm a}(t) \right). \]

Finally, we remark that by definition 
\begin{eqnarray*} 
 g(\aor{\theta}, \omega_{\theta}) -  \hat{g}(\aor{\theta}, Q_{\theta}^*)   & =&  \sum_{k} \gamma^k p_{k,\theta} \sum_{a} \pi_{\theta} (a|s) \left( \phi(s,a)^T \omega_{\theta} - Q_{\theta^*} \right) \frac{d}{d \theta} \log \pi_{\theta} (a|s) 
\end{eqnarray*} and consequently we can use Assumption \ref{ass:deltadef} to obtain that \begin{eqnarray} 
 ||g(\aor{\theta}, \omega_{\theta}) -  \hat{g}(\aor{\theta}, Q_{\theta}^*)||   & \leq &  \sum_{k} \gamma^k p_{k,\theta} \sum_{a} \pi_{\theta} (a|s) \left| \phi(s,a)^T \omega_{\theta} - Q_{\theta^*} \right| K \nonumber \\ 
 & \leq & \frac{K}{1-\gamma} E_{(s,a) \sim \nu_{\theta}} |\phi(s,a)^T \omega_{\theta} - Q_{\theta}^*(s,a)|  \leq  \frac{K}{1-\gamma} \delta \label{eq:deltabound}
\end{eqnarray} 

{\ao With all these observations in place, we next turn to the analysis of the dynamics of Proposition \ref{prop:reformulation}. 
We have structured this analysis in the following manner: 
\begin{enumerate}
    \item Section 5.3 contains technical results on regularity of the underlying functions, updates, and parameters
    \item Section 5.4 contains our convergence rate analysis and is divided to three parts:
    \begin{itemize}
        \item Section 5.4.1 focuses on the critic update;
        \item Section 5.4.2 focuses on the actor update;
        \item Section 5.4.3 includes the construction of the nonlinear small-gain theorem that, combining our previous findings, yields our main proof.
    \end{itemize}
\end{enumerate}
}

\subsection{First part of the proof: properties of the underlying functions, updates, and gradients} To analyze gradient descent one typically needs some assumptions on the underlying functions. These tend to involve continuity of various gradients and updates, boundedness and finite variance of noise, and so forth. As we have discussed in the previous section, actor-critic consists of two gradient-like updates; we will thus need to establish similar properties. We will do that in this section, and the main part of the proof itself, which will use these properties, will begin in the following section. 


We first show that the the mean direction of the critic update (before projection) is Lipschitz. Formally, we have the following. 

\begin{lemma}[{\ao Lipschitz-ness of critic update}]   There exists a constant $L_{\nabla} < \infty$ such that for all $\theta, \omega_1, \omega_2$, 
\[ ||\widetilde{\nabla}(\theta, \omega_1) - \widetilde{\nabla}(\theta, \omega_2)|| \leq L_{\nabla} ||\omega_1 - \omega_2||
\]  \label{prop:deltalip}
\end{lemma} 
\begin{proof} Indeed, by definition of $\widetilde{\nabla}(\theta, \omega)$, we have that 
\begin{eqnarray*} ||\widetilde{\nabla}(\theta, \omega_1) - \widetilde{\nabla}(\theta, \omega_2)|| & = & || A_{\theta} \omega_1 + b_{\theta} - A_{\theta} \omega_2 - b_{\theta} || \\ 
& \leq & ||A_{\theta}|| \cdot ||\omega_1 - \omega_2|| \\ 
& = & \left| \left|  E [\gamma \phi(s,a) \phi(s',a')^T - \phi(s,a)\phi(s,a)^T] \right| \right| \cdot ||\omega_1 - \omega_2||.
\end{eqnarray*} 

Now because we assumed in Assumption \ref{ass:features} that each $\phi(s,a)$ satisfies $||\phi(s,a)|| \leq 1$, we have that 
\[ \left| \left| E [\gamma \phi(s,a) \phi(s',a')^T - \phi(s,a)\phi(s,a)^T] \right| \right| \leq 2 .\] 
Thus we can take  $L_{\nabla}=2$. 
\end{proof} 

Next, at some point we will want to use the fact that none of the actor or critic updates ever leaves some compact set; while a proof can be done without this assertion, it certainly simplifies some of the arguments. To that end, our next two lemmas demonstrate that, by construction, the noises $w_{\rm c}(t), w_{\rm a}(t)$  have compact support. This turns out to be an immediate consequence of the projection of the critic update onto the compact set $\Omega$. 

 \begin{lemma}[{\ao Bounded support for critic noise}] The support of the random vector $w_{\rm c}(t)$ belongs to some compact set. \label{prop:criticsupport}
\end{lemma} 

\begin{proof} 
By definition of $w_{\rm c}(t)$, we have that 
\begin{eqnarray*}  w_{\rm c}(t) 
& = &  -(A_t - E[A_t]) \omega_t +  (b_t-E[b_t]) 
\end{eqnarray*} 
So 
\begin{eqnarray*}  ||w_{\rm c}||_2^2 & \leq & 2 ||A_t - E[A_t]||^2 ||\omega_t||^2 +  2 ||b_t - E[b_t]||^2  \\ 
& \leq & 2 (2 ||A_t||^2 + 2 ||E[A_t]||^2) ||\omega_t ||^2 + 2 (2 ||b_t||^2 + 2 ||E [b_t]||^2 ) 
\end{eqnarray*} 
Recall, however, that all costs are in $[-C_{\rm max}, C_{\rm max}]$ and that $||\phi(s,a)|| \leq 1$ for all state-action pairs $(s,a)$. We can plug this into Eq. (\ref{eq:at})   to obtain  
\[ ||w_{\rm c}(t)||_2^2 \leq 32 ||\omega_t||^2 + 8 C_{\rm max}^2
\] But since $\omega_t \in \Omega$, and $\Omega$ is a compact set, we obtain that $w_{\rm c}(t)$ indeed has compact support. 
\end{proof} 

\begin{lemma}[{\ao Bounded support for actor noise}] The support of the random vector $w_{\rm a}(t)$ belongs to a compact set.  \label{prop:wacompact}
\end{lemma} 

\begin{proof} Indeed, $w_{\rm a}(t)$ is defined as the difference between the sum on the right-hand side of Eq. (\ref{eq:pgt}) and a single term of it  chosen at random so that its expectation is the entire sum. Inspecting Eq. (\ref{eq:pgt}), it follows that because $\omega_t$ lies in the compact set $\Omega$ and, by Assumption~\ref{ass:pi}, $|| \frac{d}{d\theta} \log \pi_{\theta}(a|s)|| \leq K$ for all $\theta, s,a$, we have that $w_{\rm a}(t)$ is the difference of two quantities each of which has norm at most  $$\frac{K}{1-\gamma} ||\Phi|| \sup_{\omega \in \Omega} ||\omega||.$$ Thus its support is bounded.
\end{proof}

Since we have shown that the vectors $w_{\rm a}(t), w_{\rm c}(t)$ have bounded support, we now introduce the notation 
\[ 
E \left[ ||w_{\rm a}(t)||_2^2 ~|~ {\mathcal F}_t  \right]  \leq  \sigma_a^2, ~~
E \left[ ||w_{\rm c}(t)||_2^2 ~|~ {\mathcal F}_t \right]  \leq  \sigma_c^2. 
\]  By the above lemmas, we have that $\sigma_{\rm a} < \infty$ and $\sigma_{\rm c} < \infty$. We will also define $\sigma_{\rm a'}^2$ through
\[ 
\sqrt{E \left[ ||w_{\rm a}(t)||_2^4 ~|~ {\mathcal F}_t  \right]}  \leq  \sigma_{\rm a'}^2
\] Again, this is well-defined as $w_{\rm a}(t)$ has compact support.

Next, we need to establish that the mean of the critic update is Lipschitz. This is done in the following lemma. We remind the reader that $g(\theta, \omega)$ is defined in Eq.~\eqref{eq:gdef}. 
\begin{lemma}[\ao{Lipschitz-ness of actor update}] For all $\theta, \omega_1, \omega_2$, we have that 
\[ ||g(\theta, \omega_1) - g(\theta, \omega_2)|| \leq L_g ||\omega_1 - \omega_2 ||, \] for some $L_g < \infty$.  \label{prop:lipg}
\end{lemma}

\begin{proof} Indeed, 
\[  g(\theta, \omega_1) - g(\theta, \omega_2) = 
\sum_{s,t=0,\ldots} \gamma^t p_{t,\theta} (s) \sum_{a}  \pi_{\theta}(s,a) \phi(s,a)^T (\omega_1 - \omega_2) \frac{d}{d \theta} \log \pi_{\theta}(a|s) 
\] However, a convex combination of a collection of vectors has norm upper bounded by the largest of the norms of these vectors, so that  
\[ || g(\theta, \omega_1) - g(\theta, \omega_2) || \leq \frac{K}{1-\gamma} ||\Phi|| \cdot ||\omega_1 - \omega_2 ||  \] 

\end{proof} 

The following is a technical lemma which will be useful for us. 

\begin{lemma} Fix two vectors $u$ and $v$ and suppose $$q_{\theta} = u^T (I-\gamma P_{\theta})^{-1} v.$$ 
Then $\left| \left| \nabla_{\theta} q_{\theta} \right| \right| \leq L_q$, for some $L_q<\infty$ that does not depend on $\theta$. \label{lemm:otherinvderiv}
\end{lemma} 

\begin{proof} 
Using the well-known formula for derivative of the matrix inverse (see \cite[Eq.~(14.46)]{vinod2011hands}),  
\begin{equation} \label{eq:partderiv} \frac{\partial q_{\theta}}{\partial \theta_i} = u^T (I-\gamma P_{\theta})^{-1}  \left[ -\gamma \frac{\partial P_{\theta}}{\partial \theta_i} \right] (I - \gamma P_{\theta})^{-1} v
\end{equation} We next argue that the $2$-norm of each term in the product on the right-hand side can be bounded independently of $\theta$. Indeed, 
$||(I-\gamma P_{\theta})x||_{\infty} \geq (1-\gamma) ||x||_{\infty}$, as $||(I-\gamma P_{\theta})^{-1}||_{\infty}$ is  uniformly bounded over $\theta$ (and so is the $2$-norm).  As for the term in brackets, we have that by Assumption \ref{ass:pi} the quantities $||\frac{d}{d \aor{\theta}} \pi_{\theta}(a|s)||$ are bounded independently of $\theta$, implying any norm one takes of the term in brackets is bounded independently of $\theta$ as well. 
\end{proof} 

Next, to derive any kind of guarantee on minimizing $V$, we need to make some assumptions on this function. We already know it is smooth from the policy gradient theorem, provided that Assumption~\ref{ass:pi} holds. It is standard to have a bound asserting that its gradient cannot change too quickly.

\begin{proposition}[Smooth objective] There exists some $L_V<\infty$ such that the function $V_{\theta}$ has $L_V$-Lipschitz gradient. \label{prop:smoothvalue}
\end{proposition} 

This follows from Theorem 3 of \cite{pirotta2015policy}.

We next begin a sequence of lemmas and propositions whose ultimate goal is to show that the quantity $\omega_{\theta}$ (the ultimate limit of the TD update when $\theta$ is fixed) is a twice differentiable function of $\theta$ with a uniform upper bound on its Hessian. 

As a first step towards that, we will need to obtain upper bounds on the derivatives of the matrix $A_{\theta}$ (or rather, on the derivative of its inverse). Recall that matrix is obtained by sampling states $s_t',a_t',s_t'',a_t''$ from the distributed $\mu_{\theta}$ (see Algorithm~\ref{mainalg}). Our first proposition shows this distribution changes in a Lipschitz manner as a function of  $\theta$. We translate this into a bound on the derivatives of the entries of the matrix $A_{\theta}$.

\begin{lemma}[{\ao Uniform boundedness of derivatives and Hessians of $A_{\theta}$}] There exist constants $L_A', L_A'', L_b', L_b''$ not depending on $\theta$  such that, for all $k,l$, 
\[ 
\left| \frac{\partial [A_\theta]_{ij}}{\partial \theta_k} \right|  \leq  L_A',  
~\left| \frac{\partial^2 [A_\theta]_{ij}}{\partial \theta_k \theta_l} \right|  \leq  L_A'', 
~\left| \frac{\partial [b_{\theta}]_i}{\partial \theta_k}
\right| 
 \leq  L_b', 
~\left| \frac{\partial^2 [b_{\theta}]_i}{\partial \theta_k \theta_l}
\right| 
 \leq  L_b'.
\]  \label{lemm:abderivbounded}
\end{lemma} 

\begin{proof} Let us use the  notation $\rho_{\theta}(s,a)$ for the stationary distribution of state $s$ and action $a$ when following policy $\pi_{\theta}$. We have that 
\[ [A_{\theta}]_{ij} =\sum_{s',a',s'',a''} \rho_{\theta}(s',a') P(s''|s',a') \pi_{\theta} (a''|s'') \phi(s',a') \left( \gamma \phi(s'',a'') - \phi(s',a') \right).
\] There are only two terms in this sum depending on $\theta$. It is thus immediate that the theorem follows if we can bound the absolute values of the quantities  \begin{eqnarray*} \pi_{\theta}(a|s)  , \frac{\partial \pi_{\theta}(a|s)}{\partial \theta_j}, \frac{\partial^2 \pi_{\theta}(a|s)}{\partial \theta_j \partial \theta_k} \\ \rho_{\theta}(s,a) , \frac{\partial \rho_{\theta}(a|s)}{\partial \theta_j}, \frac{\partial^2 \rho_{\theta}(a|s)}{\partial \theta_j \partial \theta_k}
\end{eqnarray*} independently of $\theta$, for all $i,j,k$. 

Of course, both $\pi_{\theta}(a|s)$ and $\rho_{\theta}(a|s)$ are upper bounded by one. Now for the derivatives  and second derivatives of $\pi_{\theta}$, boundedness follows by Assumption \ref{ass:pi}. Finally, for the stationary 
distribution $\rho_{\theta}$, boundedness of all the quantities above  follows from \cite{heidergott2003taylor}, specifically from the discussion in the beginning of Section 4 of that paper, when put together with the earlier Remark 6 in the same paper. 

Note that here we need to use both our assumptions the uniform upper bound on the first two derivatives of $\pi_{\theta}(a|s)$ in Assumption \ref{ass:pi} as well as the uniform mixing condition of Assumption \ref{ass:stationary}, since both of these are used by \cite{heidergott2003taylor} in their arguments.

The proof of the upper bound on the derivatives and second derivatives of the entries of $b_{\theta}$ is similar. 
\end{proof}

We next establish a uniform bound on the how big $A_{\theta}^{-1}$ can get.  

\begin{lemma}[{\ao Uniform boundedness of $A_{\theta}^{-1}$}] \label{lem:invbound} 
\[ \sup_{\theta} ||A_{\theta}^{-1} || < \mu^{-1}.
\] 
\end{lemma}

\begin{proof}
We prove a more  general claim: If $\lambda_{\rm max}\frac{1}{2}(A+A^T) \leq - \lambda \leq 0$, then $||A^{-1}|| \leq \lambda^{-1}$.
By assumption, we have that
$
\frac{1}{2}x^T(A+A^T)x\leq -\lambda, 
$
for any unit vector $ x\in \mathbb{R}^n $. This implies that 
$x^Ty\leq -\lambda$, where $ y=Ax$ and~$x$ is of unit norm. On the other hand, since $\|x\|=1$, we have that
$-1 \leq \frac{x^Ty}{\|y\|}$. Using the fact that 
$-x^Ty\geq \lambda$, we conclude that $ \|y\|\geq \lambda$; since $ \lambda\geq 0 $, we have that $ x^TAA^Tx\geq \lambda^2$, for any unit vector $ x$. In particular, 
\[
\subscr{\lambda}{min}(AA^T)=\min_{\|x\|\neq 0, \|x\|=1} x^TAA^Tx \geq \lambda^2.
\]
Since $\subscr{\lambda}{max}((AA^T)^{-1})=
\subscr{\lambda}{min}(AA^T)^{-1}
$, we conclude that 
\[
\|A^{-1}\|=\sqrt{\subscr{\lambda}{max}((AA^T)^{-1})}\leq \frac{1}{\lambda}.
\]
\end{proof}

Next, we state as a lemma the fact that the limit point of TD is Lipschitz. This fact is well-known, and is an immediate consequence of the standard upper bound 
\begin{equation} \label{eq:invderiv2}
\left| \left| \frac{\partial A(\theta)^{-1}}{\partial \theta_j} \right| \right|_p \leq 
|| A(\theta)^{-1} ||_p^2 \cdot  \left| \left| \frac{\partial A(\theta)}{\partial \theta_j} \right| \right|_p
\end{equation}
for the norm of a matrix inverse (see \cite[Eq.~(14.46)]{vinod2011hands}). For a concrete reference, the next lemma is Proposition 4.4 of \cite{wu2020finite}.

\begin{lemma}[{\ao Lipschitz continuity of the TD fixed point}] \label{prop:omegalip} There is a constant $L_{\omega}$ such that 
\[ ||\omega_{\theta_1} - \omega_{\theta_2} || \leq L_{\omega} || \theta_1 - \theta_2 || \] 
\end{lemma} 

It turns out we will need a more than this; specifically, we need to argue that the second of the TD limit is independent of $\theta$. This is done in the following lemma. 

\begin{lemma}[{\ao Bounded curvature of the TD fixed point}] There is some quantity $\lambda_i$ independent of $\theta$ such that 
\[ \sup_{\theta} \lambda_{\rm max}(\nabla^2 w_{\theta}(i)) \leq \lambda_i. \] \label{lemma:twicediff}
\end{lemma} 
\begin{proof} We differentiate twice the equation $\omega_{\theta} = A_{\theta}^{-1} b_{\theta}$ to get \begin{footnotesize} 
\[ \frac{\partial^2  w_{\theta}(i)}{\partial \theta_j \partial \theta_l} = \sum_{k} \frac{\partial^2  [A(\theta)^{-1}]_{ik}}{\partial \theta_j \partial \theta_l} b_k(\theta) + \frac{\partial [A(\theta)^{-1}]_{ik}}{\partial \theta_j} \frac{\partial b_k(\theta)}{\partial \theta_l} + \frac{\partial [A(\theta)^{-1}]_{ik}}{\partial \theta_l}  \frac{\partial b_k(\theta)}{\partial \theta_j} + [A(\theta)^{-1}]_{ik} \frac{\partial^2 b_k(\theta)}{\partial^2 \theta_j}
\] \end{footnotesize} 
We need to upper bound the right-hand side independently of $\theta$. Here we can apply Lemma \ref{lem:invbound} and Lemma \ref{lemm:abderivbounded} and Eq. (\ref{eq:invderiv2}). We see that the only term not covered by these three sources is $\frac{\partial^2  [A(\theta)^{-1}]_{ik}}{\partial \theta_j \partial \theta_l}$: all the other terms have already been bounded independently of $\theta$. 

We now turn to analyzing that term.  Taking $ A_{\theta}A_{\theta}^{-1}=I$, differentiating twice and rearranging we obtain:
\[
\frac{\partial^2  A_\theta^{-1}}{\partial \theta_j \partial \theta_l}
=-A_{\theta}^{-1}
\left(
\pdertwom{A_{\theta}}{\theta_j}{\theta_l} A_{\theta}^{-1} 
+\pder{A_{\theta}}{\theta_j}\pder{A_{\theta}^{-1}}{\theta_l}
+\pder{A_{\theta}}{\theta_l}\pder{A_{\theta}^{-1}}{\theta_j}
\right)
\] The norms of all the terms on the right-hand side are bounded independently of $\theta$ as a consequence of Lemmas \ref{lem:invbound} and Lemma \ref{lemm:abderivbounded} and Eq. (\ref{eq:invderiv2}), and thus we are done. 

\end{proof}

For convenience, we define  \begin{equation} \label{eq:lambdadef} \lambda = \sqrt{\sum_{i} \lambda_{i}^2}.
\end{equation}

\subsection{Second part of the proof: small-gain analysis} 

We now turn to the main body of the proof itself. Having established that various quantities appearing in our updates are Lipschitz and/or bounded, we will now consider how the critic and actor errors relate to each other. As before, our first step is to introduce (yet more) notation.


We will  find it convenient to define  
\[ \theta_{t+1/2} = \theta_t - \beta_t g(\theta_t, \omega_t),\] so that $\theta_{t+1/2}$ is a deterministic function of $\theta_t, \aor{\omega_t}$. The next actor iterate $\theta_{t+1}$ is then obtained from $\theta_{t+1/2}$ by adding noise: 
\begin{equation} \label{eq:halfrelation} \theta_{t+1} = \theta_{t+1/2} - \beta_t w_{\rm a}(t). \end{equation}

Let us observe a couple of consequence of these equations. First, we can apply Lemma \ref{prop:omegalip} have that \begin{eqnarray}
 || \omega_{\theta_{t+1/2}} - \omega_{\theta_t} ||^2 & \leq &  L_{\omega}^2 ||\theta_{t+1/2} - \theta_t||^2  = L_{\omega}^2 \beta_t^2 ||g(\theta_t, \omega_t)||^2 \label{eq:omegadiff1}
\end{eqnarray}
Similarly, using Lemma \ref{prop:omegalip} we have that 
\begin{eqnarray}
E ||\omega_{\theta_{t+1}} - \omega_{\theta_{t+1/2}} ||^2 & \leq & E L_{\omega}^2 ||\theta_{t+1} - \theta_{t+1/2}||^2  \leq  L_{\omega}^2 \beta_t^2 \sigma_{a}^2 \label{eq:omegadiff2}
\end{eqnarray}

\subsubsection{Analysis of the Critic Update}

With the above notation and preliminaries in place, we now turn to the proof itself. Our first step is to obtain a performance error bound on the critic. Of course, this cannot be done in isolation of what happens in the actor. The final bound we will derive in this subsection will bound the critic's performance in terms of the closeness to stationary of the actor. 


Our first step is to argue that, without noise, a small enough step starting from $\omega_t$ and moving in the direction of $\widetilde{\nabla}(\theta_t, \omega_t)$ reduces the distance to $\omega_{\theta_t}$, which, recall, is the ultimate limit of the TD(0) iteration were $\theta_t$ to be left fixed.  This is stated formally, along with an estimate of the size of the reduction, in the following lemma.

\begin{lemma}[{\ao Contractivity of the TD part of the update}] If $\alpha_t \leq \mu/(2L_{\nabla}^2) $, then 
\begin{eqnarray*} ||\omega_{t} - \omega_{\theta_t} - \alpha_t \widetilde{\nabla} (\theta_t, \omega_t)||  
\leq 
 (1 - \alpha_t \mu/4)  ||\Delta_t||.
\end{eqnarray*} 
\label{lemm:splitdecrease}
\end{lemma} 
\begin{proof} Indeed, 
\begin{eqnarray*} 
||\omega_t - \omega_{\theta_t} - \alpha_t \widetilde{\nabla} (\theta_t, \omega_t)||^2 & = & 
||\omega_t - \omega_{\theta_t}||^2 - 2 \alpha_t \widetilde{\nabla}(\theta_t, \omega_t)^T (\omega_t - \omega_{\theta_t}) + \alpha_t^2 ||\widetilde{\nabla} (\theta_t, \omega_t)||^2 \\ 
& \leq & ||\omega_t - \omega_{\theta_t}||_2^2 - 2 \alpha_t \frac{\mu}{2} ||\omega_t - \omega_{\theta_t}||_2^2 + \alpha_t^2 || \widetilde{\nabla} (\theta_t, \omega_t) - \widetilde{\nabla} (\theta_t, \omega_{\theta_t})||^2 \\ 
& \leq & (1 - \alpha_t \mu + L^2_{\nabla} \alpha_t^2) ||\omega_t - \omega_{\theta_t}||^2 \\ 
& \leq & (1 - \alpha_t \mu/2) ||\omega_t - \omega_{\theta_t}||_2^2.
\end{eqnarray*} 
where the first step is just expansion of the quadratic; the second step uses Eq. (\ref{eq:thetazero}) and Eq. (\ref{eq:critupdate1}); the third step uses Lemma \ref{prop:deltalip}; and the final step uses the inequality $\alpha_t \leq \mu/(2L_{\nabla}^2)$ to get that $L_{\nabla}^2 \alpha_t^2 \leq \alpha_t \mu/2$. Taking square roots of both sides completes the proof. 
\end{proof} 

\bigskip \noindent {\bf Recursion relation for the critic.} With this lemma  in place, let us attempt to derive a recursion relation that can be used to argue that $\omega_t$ tracks $\omega_{\theta_t}$, i.e., let us prove an upper bound on the difference between these two quantities. 

Indeed, because 
\begin{eqnarray*}
    \omega_{t+1} - \omega_{\theta_{t+1}} & = & P_{\Omega} \left( \omega_t - \alpha_t \widetilde{\nabla} (\theta_t, \omega_t) - \alpha_t w_{\rm c}(t) \right)  - \omega_{\theta_{t+1}} 
\end{eqnarray*} by the non-expansiveness of projection as well by Assumption \ref{ass:projection} (which states that $\omega_{\theta} \in \Omega$ for all $\theta$), we have that
\[   ||\omega_{t+1} - \omega_{\theta_{t+1}}||^2 \leq  ||\omega_t - \alpha_t \widetilde{\nabla} (\theta_t, \omega_t) - \alpha_t w_{\rm c}(t)  - \omega_{\theta_{t+1}}||^2 
\] We will then rewrite this as 
\begin{eqnarray*}
    ||\omega_{t+1} - \omega_{\theta_{t+1}}||^2  & = &
    \left| \left|\left( \omega_t - \omega_{\theta_t} - \alpha_t \widetilde{\nabla} (\theta_t, \omega_t) - \alpha_t w_{\rm c}(t) \right) + \left( \omega_{\theta_{t+1/2}} - \omega_{\theta_{t+1}}  \right)   + \left( \omega_{\theta_t} - \omega_{\theta_{t+1/2}} \right)  \right| \right|^2
\end{eqnarray*} 

We now take the (conditional) squared expectation of both sides to obtain
\begin{footnotesize} 
\begin{align*}
   E \left[ ||\omega_{t+1} - \omega_{\theta_{t+1}}||^2 ~|~ {\mathcal F}_t \right] & =   
    E \left[ \left| \left|  \omega_t - \omega_{\theta_t} - \alpha_t \widetilde{\nabla} (\theta_t, \omega_t) - \alpha_t w_{\rm c}(t) \right| \right|^2 ~|~ {\mathcal F}_t \right] \\ & + E \left[  \left| \left| \left( \omega_{\theta_{t+1/2}} - \omega_{\theta_{t+1}}  \right)   + \left( \omega_{\theta_t} - \omega_{\theta_{t+1/2}}   \right) \right| \right|^2 ~|~ {\mathcal F}_t \right]\\ 
    &+E \left[  \left(  \omega_{\theta_{t+1/2}} - \omega_{\theta_{t+1}}    + \omega_{\theta_t} - \omega_{\theta_{t+1/2}}   \right)^T \left( \omega_t - \omega_{\theta_t} - \alpha_t \widetilde{\nabla}_t  - \alpha_t w_{\rm c}(t) \right) | {\mathcal F}_t \right]
\end{align*} 
\end{footnotesize} where, recall, ${\mathcal F}_t$ was defined in Proposition \ref{prop:reformulation}; informally, it is the history of the process up to time $t$. We then upper bound this as 
\begin{small} 
\begin{eqnarray}
   E \left[ ||\omega_{t+1} - \omega_{\theta_{t+1}}|| ~|~ {\mathcal F}_t \right]^2 & = & 
    (1 - \alpha_t \mu/4) ||\omega_t - \omega_{\theta_t}||^2 + \alpha_t^2 \sigma_{c}^2 \nonumber \\ && + 2 L_{\omega}^2 \beta_t^2 \sigma_a^2 + 2 L_{\omega}^2 \beta_t^2 ||g(\theta_t, \omega_t)||_2^2 \nonumber \\ 
        && +E \left[   \left( \omega_{\theta_{t+1/2}} - \omega_{\theta_{t+1}}   \right)^T \left( \omega_t - \omega_{\theta_t} - \alpha_t \widetilde{\nabla} (\theta_t, \omega_t) - \alpha_t w_{\rm c}(t) \right) ~|~ {\mathcal F}_t \right] \nonumber \\
    && +E \left[   \left( \omega_{\theta_t} - \omega_{\theta_{t+1/2}}   \right)^T \left( \omega_t - \omega_{\theta_t} - \alpha_t \widetilde{\nabla} (\theta_t, \omega_t) - \alpha_t w_{\rm c}(t) \right) ~|~ {\mathcal F}_t \right] \nonumber
\end{eqnarray}
\end{small} 
where in the first line we used Lemma \ref{lemm:splitdecrease}, as well as that, conditional on ${\mathcal F}_t$, $w_{\rm c}(t)$ has zero mean;  and in the second line, we used Eq. (\ref{eq:omegadiff1}) and Eq. (\ref{eq:omegadiff2}) as well as the inequality $(a+b)^2 \leq 2 a^2 + 2 b^2$.

We next turn our attention to the fourth  line of the above equation. Using  Cauchy-Schwarz, Eq.~\eqref{eq:omegadiff1}, and the conditional zero-mean of $w_{\rm c}(t)$, we obtain 
\begin{small} 
\begin{eqnarray}
   E \left[ ||\omega_{t+1} - \omega_{\theta_{t+1}}|| ~|~ {\mathcal F}_t \right]^2 & \leq & 
    (1 - \alpha_t \mu/4) ||\omega_t - \omega_{\theta_t}||^2 + \alpha_t^2 \sigma_{c}^2 \nonumber \\ && + 2 L_{\omega}^2 \beta_t^2 \sigma_a^2 +2 L_{\omega}^2 \beta_t^2 ||g(\theta_t, \omega_t)||_2^2 \nonumber \\ 
     && +E \left[   \left( \omega_{\theta_{t+1/2}} - \omega_{\theta_{t+1}}   \right)^T \left( \omega_t - \omega_{\theta_t} - \alpha_t \widetilde{\nabla} (\theta_t, \omega_t) - \alpha_t w_{\rm c}(t) \right) ~|~ {\mathcal F}_t \right] \nonumber \\ 
        &&  + \beta_t L_{\omega} ||g(\theta_t, \omega_t)|| ~~ ||\omega_t - \omega_{\theta_t} - \alpha_t \widetilde{\nabla}(\theta_t,\omega_t)||. 
    \nonumber
\end{eqnarray}
\end{small}

Finally, if $\alpha_t$ is small enough to satisfy the conditions of  Lemma \ref{lemm:splitdecrease}, we can modify the last line of this equation as 
\begin{small} 
\begin{eqnarray}
    E \left[ ||\omega_{t+1} - \omega_{\theta_{t+1}}|| ~|~ {\mathcal F}_t \right]^2 & \leq & 
    (1 - \alpha_t \mu/4) ||\omega_t - \omega_{\theta_t}||^2 + \alpha_t^2 \sigma_{c}^2 \nonumber \\ && + 2 L_{\omega}^2 \beta_t^2 \sigma_a^2 + 2 L_{\omega}^2  \beta_t^2 ||g(\theta_t, \omega_t)||_2^2 \nonumber \\ 
     && +E \left[   \left( \omega_{\theta_{t+1/2}} - \omega_{\theta_{t+1}}   \right)^T \left( \omega_t - \omega_{\theta_t} - \alpha_t \widetilde{\nabla} (\theta_t, \omega_t) - \alpha_t w_{\rm c}(t) \right) ~|~ {\mathcal F}_t \right] \nonumber \\ 
        &&  + \beta_t L_{\omega} ||g(\theta_t, \omega_t)|| ~~ ||\Delta_t|| 
    \label{eq:iter1} 
\end{eqnarray}
\end{small}

Next, we need to obtain an estimate of the size of the inner product in the third line. This is done in our next lemma.

\begin{lemma} \label{lemm:cross} Suppose $\alpha_t \leq \mu/(2L_{\nabla}^2)$. Then 
\[ E \left[ \left(\omega_{\theta_{t+1/2}} - \omega_{\theta_{t+1}} \right)^T \left(\omega_t - \omega_{\theta_t} - \alpha_t \widetilde \nabla (\theta_t, \omega_t) \right) ~|~ {\mathcal F}_t \right] \leq \beta_t^2 \frac{\lambda}{2} \sigma_{\rm a'}^2 ||\Delta_t||
\] 
\end{lemma} 

\begin{proof} Let us define $I_t$ to be the left-hand side of the above equation. Consider the quantity $\omega_{\theta}(i)$, i.e., the $i$'th entry of the vector $\omega_{\theta}$. Recall that by Lemma \ref{lemma:twicediff}, this is a twice continuously differentiable function of $\theta$ with an available upper bound of $\lambda_i$ on the norm of the Hessian at any $\theta$. 

We can thus use a second order expansion to write 
\[ \omega_{\theta_{t+1}}(i) = \omega_{\theta_{t+1/2}}(i) + \nabla \omega_{\theta_{t+1/2}}(i)^T (\beta_t w_{\rm a}(t) )  + \frac{1}{2}
\beta_t^2 w_{\rm a}(t)^T \nabla^2 \omega_{\theta_i'}(i) w_{\rm a}(t),
\] where the gradient and Hessian are taking with respect to $\theta$ and $\theta_i'$ is some point on the line segment connecting $\theta_{t+1/2}$ and $\theta_{t+1}$. 
Consequently, \begin{small}
\begin{eqnarray*} I_t & =&  E \left[ \sum_{i} \left( \beta_t \nabla \omega_{\theta_{t+1/2}}(i)^T w_{\rm a}(t) - \beta_t^2 w_{\rm a}(t)^T \frac{1}{2} \nabla^2 \omega_{\theta_i'}(i) w_{\rm a}(t) \right) \right. \\ && ~~~~~~~~\left. \left(\omega_t(s,a) - \omega_{\theta_t}(s,a) - \alpha_t \widetilde \nabla (\theta_t, \omega_t)(i) \right) ~|~ {\mathcal F}_t \right],
\end{eqnarray*} \end{small} where note that we are abusing notation somewhat, as the $\theta'$ is actually different in each term of the sum. 

Since $E[w_{\rm a}(t)|{\mathcal F}_t] = 0$ and all the quantities $w_{\rm a}(t)$ multiplies above are deterministic conditional on ${\mathcal F}_t$,  we obtain that 
\begin{small}
\[ I_t=  \beta_t^2 E \left[ \sum_{i} \left( -w_{\rm a}(t)^T \frac{1}{2} \nabla^2 \omega_{\theta_i'}(i) w_{\rm a}(t) \right) \left(\omega_t(i) - \omega_{\theta_t}(i) - \alpha_t \widetilde \nabla (\theta_t, \omega_t)(i) \right) ~|~ {\mathcal F}_t \right].
\] \end{small} 
We apply  Cauchy-Schwarz and concavity of square root to obtain 
\begin{footnotesize} 
\[ I_t \leq \beta_t^2 \sqrt{E \left[ \sum_{i}  \left( w_{\rm a}(t)^T \frac{1}{2} \nabla^2 \omega_{\theta_i'}(i) w_{\rm a}(t) \right)^2 ~|~ {\mathcal F}_t \right] } \sqrt{\sum_{i} (\omega_t(i) - \omega_{\theta_t}(i) - \alpha_t \widetilde \nabla (\theta_t, \omega_t)(i))^2}
\] 
\end{footnotesize} 

By Lemma \ref{lemma:twicediff}, we have that $w_{\rm a}(t) \nabla^2 \omega_{\theta_i'}(i) w_{\rm a}(t) 
\leq \lambda_i || w_a(t)||^2$, and plugging this in above and using Lemma~\ref{lemm:splitdecrease} completes the proof. 
\end{proof}

Having obtained this lemma, we can now go back to Eq. (\ref{eq:iter1}) and use it to bound the third line. This argument leads to the following lemma, whose punchline is that an average of the quantities $||\Delta_t||^2$ satisfies an upper bound which we will later show to be quite small. 

\begin{lemma}[{\ao Bound on the critic error}] Suppose that $\alpha_t \in (0,1), \beta_t$ are nonincreasing sequences satisfying $\alpha_{t/2}  \leq  c_{\alpha} \alpha_t$ and $\beta_{t/2}  \leq  c_{\beta} \beta_t$ for all $t$ and for some constants $c_{\alpha}, c_{\beta}$. Suppose further that for all $t \geq T/2$ we have that $24 c_{\beta} \beta_t L_{\omega} \leq 1$ and $\alpha_t \leq \mu/(2L_{\nabla}^2) $  and
\begin{equation} \label{eq:stepconstraint}  6 L_{\omega}^2 \frac{c_{\beta}^2 \beta_T^2}{\alpha_T} \frac{8}{\mu} L_g^2 + 
\frac{c_{\beta} \beta_T}{\alpha_T} \frac{2 L_\omega}{\mu} + \frac{c_{\beta} \beta_T}{\alpha_T} \frac{4}{\mu} L_\omega L_g   \leq \frac{1}{2}.
\end{equation} Then 
\begin{small}
\begin{eqnarray}
   \frac{1}{T/2} \sum_{t={T/2}}^{T-1} E \left[ ||\Delta_t|| \right]^2 & \leq & 
    \frac{1}{\alpha_{T}} \frac{\aor{16}}{ \mu T}  ||\omega_{T/2} - \omega_{\theta_{T/2}}||^2 + c_{\alpha}^2 \alpha_T  \frac{8}{\mu}  \sigma_{c}^2 \nonumber \\ && + 2 L_{\omega}^2 \frac{c_{\beta}^2 \beta_T^2}{\alpha_T}  \frac{8}{\mu}  \sigma_a^2  +  6 L_{\omega}^2 \frac{c_{\beta}^2 \beta_T^2}{\alpha_T}  \frac{8}{\mu}  \delta^2 \nonumber \\ 
        && +8 \frac{c_{\beta}^2 \beta_T^2 \sigma_{\rm a'}^2  (\lambda/2) + c_{\beta} \beta_{T} L_{\omega} \delta }{\alpha_T \mu \ }   \sqrt{ \frac{1}{T/2} \sum_{t={T/2}}^{T-1} E ||\Delta_t||^2} \nonumber \\
    && + \frac{c_{\beta} \beta_T}{\alpha_T } \frac{8 L_{\omega}}{\mu} \frac{1}{T/2} \sum_{t={T/2}}^{T-1} E ||\nabla_t||^2
     \nonumber 
\end{eqnarray} \end{small} 

\label{lemm:recur1}
\end{lemma}

\begin{proof} 
As mentioned above, our first step is to plug Lemma \ref{lemm:cross} into Eq. (\ref{eq:iter1}). We thus obtain:
\begin{eqnarray}
   E \left[ ||\omega_{t+1} - \omega_{\theta_{t+1}}|| ~|~ {\mathcal F}_t \right]^2 & = & 
    (1 - \alpha_t \mu/4) ||\omega_t - \omega_{\theta_t}||^2 + \alpha_t^2 \sigma_{c}^2 \nonumber \\ && + 2 L_{\omega}^2 \beta_t^2 \sigma_a^2 + 2 L_{\omega}^2 \beta_t^2 ||g(\theta_t, \omega_t)||^2 \nonumber \\ 
    && +\beta_t^2 \frac{\lambda}{2} \sigma_{\rm a'}^2 ||\Delta_t|| \nonumber \\ 
        && + \beta_t L_{\omega} ||g(\theta_t, \omega_t)||~ ||\Delta_t||. \label{eq:iter2}
\end{eqnarray}

We next bound the norm $||g(\theta_t, \omega_t)||$ which appears twice in the above equation. Indeed, we have that 
\begin{eqnarray*}
||g(\theta_t, \omega_t)|| 
& = &  ||g(\theta_t, \omega_{\theta_t}) + g(\theta_t, \omega_t) - g(\theta_t, \omega_{\theta_t})|| \\ & \leq & 
|| g(\theta_t, \omega_{\theta_t})|| + L_g ||\omega_t - \omega_{\theta_t}|| \\ 
& = & || \hat{g}(\theta_t, Q_{\theta_t}^*) + g(\theta_t, \omega_{\theta_t}) -  \hat{g}(\theta_t, Q_{\theta_t}^*)  || + L_g ||\Delta_t||  \\ 
& \leq & ||\nabla_t|| + \bar{\delta} +  L_g ||\Delta_t|| 
\end{eqnarray*} Here the first line simply adds and subtracts the same quantity; the second line uses Lemma \ref{prop:lipg}; the third line adds and subtracts the same quantity; the fourth line uses Eq. (\ref{eq:deltabound}) and we define 
\[ \bar{\delta} = \frac{K}{1-\gamma} \delta.\] Finally, the last line simply uses our notation  $\nabla_t = \nabla V(\theta_t)$. 

Next, using the inequality $(a+b)^2 \leq 3(a^2 + b^2 + c^2)$ 
we thus have 
\[ ||g(\theta_t, \omega_t)||^2 \leq 3 ||\nabla_t||^2 + 3 \bar{\delta}^2 + 3 L_g^2 ||\Delta_t||^2. \] Plugging this into Eq.~\eqref{eq:iter2}, we obtain 
\begin{eqnarray}
   E \left[ ||\omega_{t+1} - \omega_{\theta_{t+1}}|| ~|~ {\mathcal F}_t \right]^2 & = & 
    (1 - \alpha_t \mu/4) ||\omega_t - \omega_{\theta_t}||^2 + \alpha_t^2 \sigma_{c}^2 \nonumber \\ && + 2 L_{\omega}^2  \beta_t^2 \sigma_a^2 + 6 L_{\omega}^2 \beta_t^2 ||\nabla_t||^2 + 6 L_{\omega}^2 \beta_t^2 \bar{\delta}^2 + 6 L_{\omega}^2 \beta_t^2 L_g^2 ||\Delta_t||^2 \nonumber \\
    && + \beta_t^2  \frac{\lambda}{2} \sigma_{\rm a'}^2 ||\Delta_t|| \nonumber \\     && + \beta_t L_{\omega} \left( ||\nabla_t|| + \bar{\delta} + L_{g} ||\Delta_t|| \right) ||\Delta_t|| \nonumber \\ 
\end{eqnarray}

Using the inequality $ab \leq (1/2) (a^2 + b^2)$, taking expectations, and combining both terms containing $||\Delta_t||$, we obtain 
\begin{eqnarray}
   E \left[ ||\omega_{t+1} - \omega_{\theta_{t+1}}||  \right]^2 & = & 
    (1 - \alpha_t \mu/4) E \left[ ||\omega_t - \omega_{\theta_t}||^2  \right] + \alpha_t^2 \sigma_{c}^2 \nonumber \\ && + 2 L_{\omega}^2 \beta_t^2 \sigma_a^2 + 6 L_{\omega}^2 \beta_t^2 E \left[ ||\nabla_t||^2 \right]  + 6L_{\omega}^2  \beta_t^2 \bar{\delta}^2 + 6 L_{\omega}^2 \beta_t^2 L_g^2  E \left[ ||\Delta_t||^2 \right] \nonumber \\
    && +\left(\beta_t^2 \frac{\lambda}{2} \sigma_{\rm a'}^2  + \beta_t L_{\omega} \bar{\delta} \right) E \left[ ||\Delta_t|| \right] \nonumber \\
        && + \beta_t L_{\omega} \frac{ E \left[ ||\Delta_t||^2 \right] +E \left[ ||\nabla_t||^2 \right]}{2} + \beta_t L_{\omega}  L_g E \left[ ||\Delta_t||^2 \right] \label{eq:intermediateiter1} 
\end{eqnarray}

Our next step is to sum this up over $t={T/2}, \ldots, T-1$. To do this without excessive calculation we will use the following rough bound. Observe that if we have the recursion relation
\[ y_{t+1} \leq (1 - \zeta) y_t + e_t \] 
for some $\zeta \in (0,1)$, then we have the rough bound
\[ \sum_{t=a}^b y_t \leq \frac{y_a}{\zeta} +  \sum_{t=a}^b \frac{e_t}{\zeta}.\] 
Let us apply this to Eq. (\ref{eq:intermediateiter1}) with $\zeta = \alpha_T \mu/4$. We can do this because the sequence $\alpha_t$ is assumed to be in $(0,1)$ and nonincreasing and $\alpha_t \mu \leq 1$ since $\alpha_t \leq \mu/(2 L_{\nabla}^2)$ by assumption. Whenever the bound $\alpha_t/\alpha_T$ appears, we can upper bound it by $c_{\alpha}$, and likewise with $\beta_t/\beta_T$. Proceeding this way and applying Cauchy-Schwarz and Jensen to the fourth line, we obtain

\begin{small}
\begin{eqnarray}
   \sum_{t={T/2}}^{T-1} E \left[ ||\omega_{t} - \omega_{\theta_{t}}|| \right]^2 & \leq & 
    \frac{1}{\alpha_{T}} \frac{4}{ \mu}  ||\omega_{T/2} - \omega_{\theta_{T/2}}||^2 + c_{\alpha}^2 \alpha_T \frac{T}{2} \frac{4}{\mu}  \sigma_{c}^2 \nonumber \\ && + 2 L_{\omega}^2 \frac{c_{\beta}^2 \beta_T^2}{\alpha_T} \frac{T}{2} \frac{4}{\mu}  \sigma_a^2 +  6 L_{\omega}^2 \frac{c_{\beta}^2 \beta_T^2}{\alpha_T} \frac{8}{\mu} \sum_{t={T/2}}^{T-1} E ||\nabla_t||^2 + 6 L_{\omega}^2 \frac{c_{\beta}^2 \beta_T^2}{\alpha_T} \frac{T}{2} \frac{4}{\mu}  \bar{\delta}^2 \nonumber \\ && + 6 L_{\omega}^2 \frac{c_{\beta}^2 \beta_T^2}{\alpha_T} \frac{8}{\mu}  L_g^2 \sum_{t={T/2}}^{T-1} E ||\Delta_t||^2  \nonumber \\
        && +4 \frac{c_{\beta}^2 \beta_T^2 \sigma_{\rm a'}^2  (\lambda/2) + c_{\beta} \beta_{T} L_{\omega} \bar{\delta} }{\alpha_T \mu}  \sqrt{T/2} \sqrt{E \sum_{t={T/2}}^{T-1} ||\Delta_t||^2} \nonumber \\
    && + \frac{c_{\beta} \beta_T}{\alpha_T } \frac{2 L_{\omega}}{\mu} \sum_{t={T/2}}^{T-1} E ||\nabla_t||^2 + \frac{c_{\beta} \beta_T}{\alpha_T} \frac{2 L_{\omega} }{\mu} \sum_{t={T/2}}^{T-1} E ||\Delta_t||^2 \nonumber \\
    &&+ \frac{c_{\beta} \beta_T}{\alpha_T} \frac{4}{\mu} L_g L_{\omega} \sum_{t={T/2}}^{T-1} E  ||\Delta_t||_2^2 
     \nonumber 
\end{eqnarray} \end{small}

Let us observe that the coefficient of $\sum_{t=T/2}^{T-1} ||\Delta_t||^2$ on the right-hand side is 
\[ 6 L_{\omega}^2 \frac{c_{\beta}^2 \beta_T^2}{\alpha_T} \frac{8}{\mu} L_g^2 + 
\frac{c_{\beta} \beta_T}{\alpha_T} \frac{2 L_\omega}{\mu} + \frac{c_{\beta} \beta_T}{\alpha_T} \frac{4}{\mu} L_\omega L_g   \leq \frac{1}{2},
\] where the inequality is by assumption in Eq. (\ref{eq:stepconstraint}). We thus have

\begin{small}
\begin{eqnarray}
   \frac{1}{2} \sum_{t={T/2}}^{T-1} E \left[ ||\omega_{t} - \omega_{\theta_{t}}|| \right]^2 & \leq & 
    \frac{1}{\alpha_{T}} \frac{4}{ \mu}  ||\omega_{T/2} - \omega_{\theta_{T/2}}||^2 + c_{\alpha}^2 \alpha_T \frac{T}{2} \frac{4}{\mu}  \sigma_{c}^2 \nonumber \\ && + 2 L_{\omega}^2 \frac{c_{\beta}^2 \beta_T^2}{\alpha_T} \frac{T}{2} \frac{4}{\mu}  \sigma_a^2  +  6 L_{\omega}^2 \frac{c_{\beta}^2 \beta_T^2}{\alpha_T} \frac{T}{2} \frac{4}{\mu}  \bar{\delta}^2 \nonumber \\ 
        && +4 \frac{c_{\beta}^2 \beta_T^2 \sigma_{\rm a'}^2  (\lambda/2) + c_{\beta} \beta_{T} L_{\omega} \bar{\delta} }{\alpha_T \mu}  \sqrt{T/2} \sqrt{E \sum_{t={T/2}}^{T-1} ||\Delta_t||^2} \nonumber \\
    && + \frac{c_{\beta} \beta_T}{\alpha_T } \frac{4 L_{\omega}}{\mu} \sum_{t={T/2}}^{T-1} E ||\nabla_t||^2,
     \nonumber 
\end{eqnarray} \end{small}

where we also combined two different terms that multiply $\sum_{t=T/2}^{T-1} E ||\nabla_t||^2$ using the assumption that $24 c_{\beta} \beta_{T} L_{\omega}  \leq 1$. Finally, dividing by $T/4$ we obtain the  statement of the lemma.

\end{proof} 



\subsubsection{Analysis of the actor update} 

We have just given an analysis of the performance of the critic, and that analysis had an error term corresponding to the actor stationarity, i.e., to the quantities $||\nabla_t||$. We now give an analysis of the reverse, i.e., a performance bound on the actor in terms of the critic error. 

\begin{lemma}[{\ao Bound on the actor error}] Suppose $\beta_t \leq 1/(6 L_V)$ for all $ t \geq T/2$. Then \begin{footnotesize} 
\[ \frac{1}{T/2} \sum_{t=T/2}^{T-1} E ||\nabla_t||^2 \leq 8 \left( \frac{ V(\theta_{T/2}) - V(\theta_T)}{\beta_T T} +  c_{\beta} L_g^2 \frac{1}{T/2} \sum_{t=T/2}^{T-1} E ||\Delta_t||^2 +  c_{\beta}  \bar{\delta}^2 +  c_{\beta}^2 \beta_T \frac{3}{4} L_V \sigma_{\rm a}^2 \right) \] \end{footnotesize} \label{lemm:gradgain}
\end{lemma} 

\begin{proof} We will attempt to analyze the actor update as an inexact gradient descent. Our first step is to upper bound  the difference between the expected actor update direction $g(\theta_t, \omega_t)$ and the true gradient $\nabla V(\theta_t) = \hat{g}(\theta_t, Q_{\theta_t}^*)$. To do this, we argue as
\begin{eqnarray*} || g(\theta_t, \omega_t) - \hat{g}(\theta_t, Q_{\theta_t}^*)|| & \leq & ||g(\theta_t, \omega_t) - g(\theta_t, \omega_{\theta_t}) || + || g(\theta_t, \omega_{\theta_t}) - \hat{g}(\theta_t, Q_{\theta_t}^*) || \\ 
& \leq & L_g ||\Delta_t|| + \bar{\delta},
\end{eqnarray*}  where the bound on the first term comes from Lemma \ref{prop:lipg} while the bound on the second term comes from  Eq. (\ref{eq:deltabound}) where, recall that $\bar{\delta} = (K/(1-\gamma)) \delta.$  Consequently, we can write the actor update as 
\begin{equation} \label{eq:thetadiff} \theta_{t+1} = \theta_t - \beta_t \nabla V(\theta_t) - \beta_t d_t - \beta_t w_{\rm a}(t),
\end{equation} where \begin{equation} \label{eq:dbound} ||d_t|| \leq   L_g ||\Delta_t|| + \bar{\delta}, 
\end{equation}

By Proposition \ref{prop:smoothvalue}, the function $V(\theta)$ has $L$-Lipschitz gradient. By the ``descent lemma'' (e.g., Theorem 4.22 in \cite{beck2014introduction}) it satisfies 
\[ V(\theta_{t+1}) \leq V(\theta_t) + \nabla V(\theta_t)^T (\theta_{t+1} - \theta_t) + \frac{L_V}{2} ||\theta_{t+1} - \theta_t||^2. \] Plugging this into Eq. (\ref{eq:thetadiff}), taking expectations, and using the inequality $$(\sum_{i=1}^k |a_i|)^2 \leq k \sum_{i=1}^k a_i^2,$$ we obtain
\begin{eqnarray*} E V(\theta_{t+1}) & \leq &  E V(\theta_t) - (\beta_t/2)  E ||\nabla V(\theta_t)||^2 + \frac{\beta_t}{2}E  ||d_t||^2  \\ 
&& + 3 \frac{L_V \beta_t^2}{2} \left(  E ||\nabla V(\theta_t)||^2 + E ||d_t||^2 +  \sigma_a^2 \right)
\end{eqnarray*} Now as a consequence of our assumptions, we have that $(3/2) L_V \beta_t^2 \leq \beta_t/4$ so that we have 
\[ E V(\theta_{t+1}) \leq E V(\theta_t) - (\beta_t/4) E || \nabla V(\theta_t)||^2 + \beta_t E ||d_t||^2 + \frac{3}{2} L_V \beta_t^2 \sigma_{\rm a}^2.\] Now Eq. (\ref{eq:dbound}) allows us to rewrite this as 
\[ E V(\theta_{t+1}) \leq E V(\theta_t) - (\beta_t/4) E ||\nabla_t||_2^2 + 2 \beta_t L_g^2 E ||\Delta_t||^2 +2 \beta_t  \bar{\delta}^2 +  \frac{3}{2} L_V \beta_t^2 \sigma_{\rm a}^2. \] We re-arrange this as 
\[  (\beta_t/4) E ||\nabla_t||_2^2 \leq E V(\theta_t) - E V(\theta_{t+1})  + 2 \beta_t L_g^2 E ||\Delta_t||^2 +2 \beta_t  \bar{\delta}^2 + \frac{3}{2} L_V \beta_t^2 \sigma_{\rm a}^2. \] Summing this over $t=T/2, \ldots, T-1$ we obtain \begin{footnotesize}
\[ \sum_{t=T/2}^{T-1} \frac{\beta_t}{4}  E ||\nabla_t||^2  \leq  E V(\theta_{T/2})  - E V(\theta_{T}) + 2 L_g^2 \sum_{t=T/2}^{T-1} \beta_t E ||\Delta_t||^2 +  2   \bar{\delta}^2   \sum_{t=T/2}^{T-1} \beta_t +  \frac{3}{2} L_V \sigma_{\rm a}^2 \sum_{t=T/2}^{T-1} \beta_t^2 \] \end{footnotesize} Dividing by $(\beta_T/4)(T/2)$ concludes the proof. 
\end{proof}

\subsubsection{A nonlinear small-gain theorem} 

We have just finished deriving two relations, one analyzing how the actor error affects the critic error (Lemma \ref{lemm:recur1}), and the other one analyzing how critic error affects actor error (Lemma \ref{lemm:gradgain}).
We now need to put these two bounds together. This is done using the following nonlinear version of the small gain theorem. 

\begin{lemma}[{\ao Small-gain bound}] \label{lemm:sqrt} 
Suppose $x_1, x_2, \ldots$ and \aor{$z_1, z_2, \ldots$}, are  sequences of vectors and  $||\cdot||_T$ is a semi-norm on the first $T$ element of a vector sequence. If there exist nonnegative numbers $a, b,c,d, e$ such that 
\begin{eqnarray}  ||x||_{T} & \leq & a + b \sqrt{||x||_T} + c ||z||_T, \label{eq:firsteq} \\ 
 ||z||_T & \leq & d + e ||x||_T \label{eq:secondeq}
 \end{eqnarray} and $2 ce < 1$, then for all $T$, 
\begin{eqnarray*} ||x||_T \leq \frac{ 2 a + b^2 + 2cd}{1-2ce}  
\end{eqnarray*} 
\end{lemma} 

Note that this differs  from the usual small-gain theorem due to the square root in the first equation.  This difference is what forces us to put an extra factor of two into the gain bound. In words, what this lemma says is that the introduction of a square root into one of the small gain bounds adds a term to the final bound which is quadratic in the variable $b$ multiplying the square root. The proof is given next.

\smallskip

\begin{proof} Let us substitute $||x||_T = y^2$ where $y$ is also nonnegative. Then 
\[ y^2 \leq a + b y  + c ||z||_T \] or 
\[ y^2 - b  y - (a + c||z||_T) \leq 0 \] 
The set where a quadratic with a leading coefficient of $1$ is nonpositive is always the interval between its roots. This we can upper bound $y$ by the largest root of this quadratic: 
\[ y \leq \frac{b + \sqrt{b^2 + 4 (a+c||z||_T)}}{2}. \] Squaring both sides and using the inequality $(u+v)^2 \leq 2(u^2 + v^2)$, 
\[ y^2 \leq \frac{2 b^2 + 2 (b^2 + 4 (a + c ||z||_T))}{4}, \]
and simplifying and using $||x||_T=y^2$, we obtain. 
\begin{equation} \label{eq:tempeq} ||x||_T \leq 2 a + 2 c ||z||_T + b^2.\end{equation}

We now plug  Eq. (\ref{eq:secondeq}) into this to obtain that 
\[ ||x||_T \leq 2a + b^2 + 2c(d + e ||x||_T), 
\] so that 
\[ ||x||_T \leq \frac{ 2 a + b^2 + 2cd}{1-2ce} \]  

\end{proof} 

We now use this nonlinear small-gain theorem we have derived to put together our bounds on actor and critic. Specifically, we will simply combine Lemma \ref{lemm:recur1} giving us a performance bound on the critic with Lemma \ref{lemm:gradgain} on the actor. Inspecting these two lemmas, we see that their form is exactly the form one needs to apply the small gain theorem above. 

\smallskip

\begin{proof}[Proof of Theorem \ref{thm:iid} ({\ao our main result})] To apply the  small gain theorem just derived, we will plug in $x_t = \Delta_t$ and $y_t = \nabla_t$. Naturally, the semi-norm we have derived will be $||u||_{T} = \frac{1}{T/2} \sum_{t=T/2}^{T-1} u_t^2.$ Inspecting Lemma \ref{lemm:recur1} and Lemma \ref{lemm:gradgain}, we can pattern match the quantities $a,b,c,d,e$ to the quantities from those equations, and obtain  

\begin{small}
\begin{eqnarray}
   \frac{1}{T/2} \sum_{t={T/2}}^{T-1} E \left[ ||\Delta_t|| \right]^2 & \leq & 
    \frac{1}{\alpha_{T}} \frac{\aor{64}}{ \mu T}  ||\omega_{T/2} - \omega_{\theta_{T/2}}||^2 + c_{\alpha}^2 \alpha_T  \frac{32}{\mu}  \sigma_{c}^2 \nonumber \\ && + 4 L_{\omega}^2 \frac{c_{\beta}^2 \beta_T^2}{\alpha_T}  \frac{16}{\mu}  \sigma_a^2  +  12 L_{\omega}^2 \frac{c_{\beta}^2 \beta_T^2}{\alpha_T}  \frac{16}{\mu}  \aor{\delta}^2 \nonumber \\ 
        && +2 \frac{2 c_{\beta}^4 \beta_T^4 \sigma_{\rm a'}^4  (\lambda/2)^2 + 2 c_{\beta}^2 \beta_{T}^2 L_{\omega}^2 \aor{\delta}^2 }{\alpha_T^2 \mu^2  }    \nonumber \\
  && + \frac{c_{\beta} \beta_T}{\alpha_T } \frac{16 L_{\omega}}{\mu} 16 \left( \frac{ V(\theta_{T/2}) - V(\theta_T)}{\beta_T T} +   c_{\beta}  \bar{\delta}^2 +  c_{\beta}^2 \beta_T \frac{3}{4} L_V \sigma_{\rm a}^2\right)
     \label{eq:jgain2323} 
\end{eqnarray} \end{small} 
 
 For this to be a valid step, we need to make sure the gain condition of $1-2ec>\aor{1/2}$ is satisfied. For that, we need to choose  $\beta_{T} = c' \alpha_T$ for a small enough $c'$. Then the small gain condition $2ec<\aor{1/2}$ is equivalent to $  c' \leq \frac{\mu}{8L_\omega}  \frac{1}{4 c_{\beta}^2 L_g^2}.$

 We can now do the final step-size selection step.  First, we can choose $\alpha_t = 1/\sqrt{t}$ and $\beta_t = c'/\sqrt{t}$, where $c'$ will be chosen to make this step-size satisfy all the assumptions we have imposed throughout. Let us now recall what those are. 
 
 First, we have assumed $\alpha_t \leq \mu/(2L_{\nabla}^2) $ and  $\beta_t \leq 1/(6 L_V)$ and \aor{$24 c_{\beta} \beta_t L_{\omega} \leq 1$}. However, these need to hold for $t \geq T/2$. Thus any step-size choices $\alpha_t, \beta_t$ which decrease to zero will eventually satisfy these. 
 
 However, we have two equations that have imposed relationships between the step-sizes $\alpha_t$ and $\beta_t$: these are $  c' \leq \frac{\mu}{8L_\omega}  \frac{1}{4 c_{\beta}^2 L_g^2}.$ and Eq. (\ref{eq:stepconstraint}). Considering Eq. (\ref{eq:stepconstraint}), for large enough $T$, the first term of that equation will be negligible compared to the second and third, and so to satisfy it we need to choose $c'  = O \left( \max \left(  \frac{\mu}{c_{\beta} L_{\omega}}, \frac{\mu}{L_{\omega} L_g c_{\beta}} \right) \right).$  We conclude there is a particular choice of $c$ depending on $c_{\beta}, \mu, L, L_{\omega}, L_g$ that makes all of the step-size assumptions satisfied. Further, treating all variables except $\mu$ as constant, we have that $c' = O(\mu)$. 

Let us next examine Eq. (\ref{eq:jgain2323}) to see what convergence rate we can obtain in that scenario. We immediately see that every term  except for the $O(\bar{\delta}^2)$ terms, which gets multiplied by constant factors, decays like $1/\sqrt{T}$ or faster (we use that $||\omega_{T/2} - \omega_{\theta_{T/2}}||$ is bounded independently of $T$, because these terms are in the compact set $\Omega$, and likewise  $||V_{\theta_{T/2}}-V_{\theta_T}||$ is bounded independently of $T$ as the value of any policy lies in $C_{\rm max}/(1-\gamma)$). Thus the right-hand side of Eq. (\ref{eq:jgain2323}) is $O(\bar{\delta}^2 + \mu^{-1}/\sqrt{T})$. We then plug this into Lemma \ref{lemm:gradgain} to obtain the same bound for the average of the last $T/2$ quantities $||\nabla_t||^2$.

\end{proof}

 \begin{remark} {\ao We briefly revisit the issue of time-scale separation in the context of our proof above. Recall that, in the Introduction, we discussed how slowing down the actor relative to the critic would result in worse convergence rate. Let us see how this issue appears in the context of our small-gain proof.} 
 
 {\ao Indeed, note that, immediately after Eq. (\ref{eq:jgain2323}), we choose $\alpha_t$ and $\beta_t$ both to be proportional to $1/\sqrt{t}$. To see why we make this choice, observe Eq. (\ref{eq:jgain2323}) contains terms that scale as $\alpha_T, (\alpha_T T)^{-1}$, which force the choice $\alpha_T \sim 1/\sqrt{T}$ to get the best  decay rate of $T^{-1/2}$ that one can get from Eq. (\ref{eq:jgain2323}). Likewise, Lemma \ref{lemm:gradgain} contains terms that scale as $\beta_T, (\beta_T T)^{-1}$, forcing us to make a similar choice for $\beta_t$ (fortunately, this pair of choices results in all other terms of those bounds decaying at $T^{-1/2}$ or faster).} 
 
 {\ao Thus, consistent with our earlier discussion, we see that a two-timescale approach which artificially slows down the actor by choosing $\alpha_t$ to decay at an asymptotically faster rate than $\beta_t$ would result in a worse final convergence rate because at least one of the four terms $\alpha_T, (\alpha_T T)^{-1}, \beta_T, (\beta_T T)^{-1}$ would decay slower than $T^{-1/2}$.} 
\end{remark} 

 \medskip
 
\begin{remark} {\ao It is usually possible to take a small gain argument and convert it to an explicit Lyapunov function, and it is natural to wonder if one could do the same here. Ideally, one would write down an explicit Lyapunov function $V$ which converges to an approximate stationary point of actor-critic at a rate of $\mu/\sqrt{t}$ as in our main result. We would conjecture that this is indeed possible. The main difficulties in establishing this lie in the nonlinear nature of the small-gain theorem we use here and the step-size-dependent discrete nature of the updates. }
\end{remark} 

\begin{remark} {\ao The small gain theorem is extremely widely used in control. The core idea of it is applicable to other settings where convergence of interconnected of dynamical systems is under consideration. Such couplings are naturally present in the context of actor-critic.} 

{\ao The above application of small gain theorem is slightly unusual, since the ``typical'' application in control requires a careful bound on the input-output gain of a dynamical system, usually involving some kind of frequency analysis. Because the recursions here are relatively simpler from the dynamical point of view, frequency analysis turns out to be unnecessary, and we can attempt to directly bound the gains from critic-error to actor-error and vice versa. The argument presented here is in the spirit of the previous work of one of the authors in \cite{nedic2017achieving}.}
\end{remark}

\section{Conclusion} We  showed how to analyze single timescale actor critic using a nonlinear version of the small gain theorem, resulting in an improved sample complexity of $O\left(\mu^{-2} \epsilon^{-2}\right)$. Our work suggests a number of future directions.

First, we have assumed that the critic uses a linear approximation. Generalizing these results to even some tractable classes of nonlinear approximations is likely to be challenging. However, there are a number of theoretical results suggesting that neural networks in the ``large width'' regime exhibit nonlinearities which are bounded (in the sense of the gradient having a small Lipschitz constant; for a linear function, the Lipschitz constant of the gradient is zero), and consequently may be analyzable using the small-gain approach we have used here. 

Second, one may be able to improve the sample complexity still further. However, we believe that such an improvement would most likely require a more complex algorithm, perhaps using momentum and acceleration, and at the very least going beyond the sort of classic actor-critic update we have studied. Moreover, because accelerated methods are less robust to to errors, the analysis of such an accelerated scheme could be challenging. 

Third, there are a number of popular variations of actor critic, such a ``natural actor critic'' obtained by adding regularization terms. In many applications, such methods are known to significantly outperform  actor critic. It may be possible to modify our analysis to give a rigorous analysis of those variations. As discussed in our literature review section, as of writing the best rate for natural actor critic is $O(\epsilon^{-3})$.

Finally, it would be interesting to specialize these rates further to problems of epidemic stabilization, where there is some structure in the underlying RL problem that can be exploited to obtain faster rates, e.g., using the model from \cite{ma2020optimal}.

\bibliographystyle{siamplain}
\bibliography{mybib}
\end{document}